\documentclass[10pt, letterpaper,reqno]{amsart}
\usepackage{amssymb}
\usepackage{amsmath}
\usepackage{amsfonts, color}
\usepackage{graphicx}
\usepackage{mathrsfs,amssymb, slashed, cite}
\usepackage{tikz}
\usetikzlibrary{patterns, angles, quotes} 

\usepackage{hyperref}
\usepackage{cleveref}

\newcommand{\R}{\mathbb{R}}
\newcommand{\C}{\mathbb{C}}

\newcommand{\supp}{\text{supp}}

\newcommand{\Z}{\mathbb{Z}}

\newtheorem{corollary}{Corollary}[section]

\newtheorem{definition}{Definition}

\newtheorem{lemma}{Lemma}[section]

\newtheorem{proposition}{Proposition}[section]

\newtheorem{theorem}{Theorem}[section]
\numberwithin{equation}{section}

\newcommand{\qtq}[1]{\quad\text{#1}\quad}

\begin{document}

\title[Scattering for NLS with inhomogeneous nonlinearities]{Scattering for the $2d$ NLS with inhomogeneous nonlinearities.}

\author[L. Baker]{Luke Baker}

\begin{abstract}  We prove large-data scattering in $H^1$ for inhomogeneous nonlinear Schr\"odinger equations in two space dimensions for all powers $p>0$. We assume the inhomogeneity is nonnegative and repulsive; we additionally require decay at infinity in the case $0<p\leq 2$.  We use the method of concentration-compactness and contradiction. We preclude the existence of compact solutions using a Morawetz estimate in the style of Nakanishi \cite{N}, adapted to repulsive inhomogeneous nonlinearities in \cite{BM}.
\end{abstract}

\maketitle

\section{Introduction}\label{S:Intro}

In this work we consider the large data scattering theory for inhomogeneous nonlinear Schr\"odinger equations of the form
\begin{equation}\label{nls}
\begin{cases}
i\partial_tu+\Delta u=a(x)|u|^pu,\\
u|_{t=0}=  u_0\in H^1(\R^2)  
\end{cases}
\end{equation}

Here $u:\R_t\times\R_x^2\to \C$ and $a:\R^2\to \R$. Such equations enjoy conservation of \textit{mass} and \textit{energy} defined by
\begin{equation}\label{E:MassEnergy}
M(u)=\int|u(x)|^2\,dx\qtq{and}E(u)=\int\frac{1}{2}|\nabla u|^2+\frac{a(x)}{p+2}|u|^{p+2}\,dx,
\end{equation}
respectively. We will not consider arbitrary weights $a$, but restrict ourselves to the following class of functions:

\begin{definition}\label{D:ADef}
Let $p>0$. We say that $a:\R^2\to \R$ is \textit{admissible} if it is nonnegative, repulsive in the sense that $x\cdot \nabla a(x)\leq 0$, and satisfies
\[
\begin{cases}
a, \nabla a\in L^\infty(\R^2) & p>2\\
a,\nabla a\in L^1(\R^2)\cap L^\infty(\R^2) & p\leq 2.
\end{cases}
\]
Under these conditions, we may obtain a map $\tilde{a}:\mathbb{S}^1\to \R$ by setting
\begin{equation}\label{E:ATilde}
\tilde{a}(\theta)=\lim_{r\to\infty}a(r\cos(\theta),r\sin(\theta)).
\end{equation}
For $a$ to be admissible, we additionally require that $\tilde{a}$ is continuous.
\end{definition}

With this definition in place our main result is the following:

\begin{theorem}\label{T}
Let $p>0$ and let $a:\R^2\to \R$ be admissible. Then for any $u_0\in H^1(\R^2)$, there exists a unique global solution $u:\R\times\R^2\to \C$ to \eqref{nls} with $u|_{t=0}=u_0$. Moreover, $u$ scatters in $H^1$ as $t\to\pm\infty$; that is, there exists $u_\pm\in H^1(\R^2)$ such that
\[
\lim_{t\to\pm\infty}\|u(t)-e^{it\Delta}u_\pm\|_{H^1(\R^2)}=0,
\]
where $e^{it\Delta}$ is the free Schr\"odinger group.
\end{theorem}

The case of $a\equiv1$ corresponds to the standard defocusing power-type NLS, a widely studied model, for which an $H^1$ scattering theory with $H^1$ initial data is available for exponents $p$ satisfying $\frac{4}{d}\leq p\leq \frac{4}{d-2}$. The cases $p=\frac{4}{d}$ and $p=\frac{4}{d-2}$ are referred to as \textit{mass-critical} and \textit{energy-critical} respectively, and the range $\frac{4}{d}<p<\frac{4}{d-2}$ is referred to as the \textit{intercrtical regime}. We refer the reader to \cite{GV, N} for the intercritical results, to \cite{CKSTT, RV, V2} for the energy-critical results, and to \cite{D, D2, D3} for the mass-critical results. Additionally, for $p\leq \frac{2}{d}$ it is known that scattering fails for any nonzero $u_0\in\Sigma=\{f\in H^1:xf\in L^2\}$, see \cite{S}. We briefly remark that one can obtain $H^1$ scattering in the so called \textit{short range case} $\frac{2}{d}<p<\frac{4}{d}$ if one works with $u_0\in \Sigma$, see \cite{TY}. Finally, for an expository paper that discusses these results in more detail we refer the reader to \cite{M-expo}.

The key to understanding these results is to observe that in the case of the standard power-type NLS one has a scaling symmetry. Namely, if $u$ solves \eqref{nls} with $a\equiv1$, then for $\lambda>0$ the function $u_\lambda$ defined by
\[
u_\lambda(t,x)=\lambda^\frac{2}{p}u(\lambda^2t,\lambda x),
\]
also solves \eqref{nls} and in fact one has
\[
\|u_\lambda\|_{L_t^\infty \dot{H}^{s_c}_x(\R\times\R^d)}=\|u\|_{L_t^\infty \dot{H}^{s_c}_x(\R\times\R^d)},\quad s_c=\frac{d}{2}-\frac{2}{p}.
\]
Thus the value of $p$ provides a notion of criticality for the problem.

In the case of a more general weight $a\not\equiv1$ this scaling symmetry may be lost and so the correct notion of criticality is less clear. For some models (e.g. the \textit{inhomogeneous NLS} where $a(x)=|x|^{-b}$) there is still a scaling symmetry and so this problem is easily resolved, however in the generality we consider this is not the case (cf. \cite{MMZ, CFGM, M-INLS, FarahG, CC, CLZ, CLZ2, GM, AGT, LMZ}).

In \cite{BM}, the authors considered the general inhomogeneous case when $d=1$, and were able to prove scattering in the restricted case $p>2$. There the following observation was made: In the case that $a\equiv1$, one wants to estimate the nonlinear term $|u|^pu$ in some dual Strichartz norm, i.e. in a space of the form $L_t^{\tilde \alpha'}L_x^{\tilde \beta '}$ for some admissible pair $(\tilde\alpha,\tilde\beta)$. By H\"older's inequality one has
\[
\||u|^pu\|_{L_t^{\tilde\alpha'}L_x^{\tilde\beta'}}\leq \|u\|_{L_t^4L_x^\infty}\|u\|^p_{L_t^{q}L_x^r},
\]
where the space $L_t^qL_x^r$ scales critically, i.e.
\[
\frac{2}{q}+\frac{d}{r}=\frac{2}{p}=\frac{d}{2}-s_c.
\]
Analogously, the authors of \cite{BM} defined a notion of criticality in the case $a\not\equiv1$. Indeed again estimating the nonlinearity $a|u|^pu$ in the dual Strichartz space $L_t^\frac{4}{3}L_x^1$ and placing $a$ in some space $L^\rho$ one obtains the nonlinear estimate
\[
\|a|u|^pu\|_{L_t^\frac{4}{3}L_x^1}\leq \|a\|_{L_x^\rho}\|u\|_{L_t^4L_x^\infty}\|u\|_{L_t^{q}L_x^r}^p,
\]
where $q=2p$ and $r=\frac{\rho p}{\rho -1}$. Thus in the $1d$ case one is able to treat the problem as $\dot{H}^s$-critical where
\[
s=\frac{1}{2}-\frac{2}{q}-\frac{1}{r}.
\]
In particular if $\rho=1$ then $r=\infty$ and one has
\[
s=\frac{1}{2}-\frac{1}{p},
\]
and so under this notion of criticality, the problem can be treated as intercritical as long as $p>2$ (by choosing $a\in L^\rho$ appropriately).

In dimension $d=2$ if one performs a similar computation one has
\[
\|a|u|^pu\|_{L_t^2L_x^ 1}\leq \|a\|_{L_x^\rho}\|u\|_{L_t^2L_x^\infty}\|u\|_{L_t^\infty L_x^r}^p
\]
where $r=\frac{\rho p}{\rho -1}$. We remark that while the above estimate is in general forbidden (due to the failure of endpoing Strichartz estimates in $2d$) it is nonetheless useful for working out the scaling. In particular, borrowing the notion of criticality from the $1d$ case one can treat the problem as $\dot{H}^s$-critical where 
\[
s=\frac{2}{2}-\frac{2}{r}=1-\frac{1}{p}+\frac{1}{\rho p}
\]
We see that by choosing $\rho$ close enough to $1$, for any $p>0$ we can always enforce $s$ strictly between $0$ and $1$.  Thus by imposing decay on $a$ one can treat the problem as intercritical for any $p>0$. 

The exact values of $\rho$ we choose are
\begin{equation}\label{E:RhoValue}
\rho=\rho(p)=\begin{cases}
\infty & p>2\\
\frac{4}{4-p} & p\leq 2,
\end{cases}
\end{equation}
so that the corresponding critical scaling $s$ is
\begin{equation}\label{E:ScalValue}
s(p)=\begin{cases}
1-\frac{2}{p}& p>2\\
\frac{1}{2} & p\leq 2.
\end{cases}    
\end{equation}

We briefly remark that in dimensions $d>2$ similar considerations also show that for any $0<p<\frac{4}{d-2}$ the problem may be treated as intercritical. Indeed, establishing scattering in this case (under slightly stronger asumptions on $a$) is simpler as one has access to Lin-Strauss type Morawetz estimates (cf. \cite{W}).

While this scaling heuristic suggests that scattering should be possible for all $p>0$, treating the full range of $p>0$ introduces new challenges compared to the $1d$ case. In particular, the case of $p<1$ makes developing a suitable stability theory more difficult. To get around this issue we follow the approach of \cite{HR, M-half} utilizing exotic Strichartz estimates and working with spaces that scale critically, but which do not involve any fractional derivatives.

Another significant challenge introduced by moving to $2d$ arises when carrying out the concentration-compactness and contradiction method. More specifically, we will need to construct scattering solutions to \eqref{nls} corresponding to the profiles arising from the linear profile decomposition, Proposition~\ref{P:LPD},  living far from the origin. Due to the broken translation symmetry this is not so simple. Indeed, in the case $p>2$ an admissible $a$ can have different limits as $|x|\to \infty$ along different directions, so that there is no clear limiting model. In the $1d$ setting this issue is not so severe as $a$ can have at most two different behaviors and the existing methods for approximation can be readily adapted. However, in $2d$ some new constructions are needed before the existing methods can be properly incorporated (cf. \cite{BM, KMVZZ, KMVZ, MMZ, CFGM}). The details may be found in Section~\ref{S:SFAP2}.

In the rest of the introduction we briefly discuss the proof of Theorem~\ref{T}.

First, under the assumptions of Theorem~\ref{T} we may establish a suitable local well-posedness theory using the standard arguments. This may then be upgraded to global well-posedness using conservation of mass and energy. Additionally, for solutions $u$ such that $M(u)+E(u)$ is sufficiently small, its not difficult to establish scattering.

Thus if Theorem~\ref{T} were to fail, there would be some mass-energy threshold, $E_c$, such that all solutions $u$ to \eqref{nls} satisfying $M(u)+E(u)<E_c$ scatter, while at and above this threshold scattering may fail. We further show that in this scenario, we may construct a minimal non-scattering solution $u$ to \eqref{nls} such that $M(u)+E(u)=E_c$. Because of its minimality, this non-scattering solution will additionally have the property that its orbit is pre-compact in $H^1$. The general approach to constructing such solutions is standard at this point (using linear and nonlinear profile decompositions); however, as mentioned above, the broken translation symmetry leads to new challenges in constructing the nonlinear profiles corresponding to profiles living far from the origin.

Theorem~\ref{T} is therefore reduced to precluding the existence of compact solutions as described above. We carry this out in Section~\ref{S:Preclusion} using a Morawetz estimate introduced by Nakanishi in \cite{N}. Under the repulsivity assumption on $a$, the authors of \cite{BM} showed that the same estimate could be adapted to the case of \eqref{nls}. Ultimately, we show that such Morawetz estimates are incompatible with a solution having a pre-compact orbit in $H^1$, thus establishing the final result. We remark that for the standard power-type NLS in dimensions $d=1,2$ one can obtain a simpler proof of scattering due to the existence of \textit{interaction Morawetz estimates} (cf. \cite{CGT, CHVZ}). However, such estimates rely crucially on the spatial translation symmetry of the equation, and thus are not available in our setting, see \cite{BM} for further discussion.

For the standard power-type NLS the concentration compactness methods are not needed to prove $H^1$ scattering in the intercritical regime, due to \textit{interaction Morawetz estimates} \cite{CGT, CHVZ}. Such estimates rely crucially on the translation symmetry for the power-type NLS and are thus not available in the case of inhomogeneous nonlinearities. See \cite{BM} for further discussions.

The rest of the paper is organized as follows: In Section~\ref{S:prelim} we introduce the notation and basic results used throughout the rest of the paper. In Section~\ref{S:Stability} we review the well-posedness theory for \eqref{nls} and establish a stability result. In Section~\ref{S:SFAP1} we prove a scattering for far-away profiles result in the simpler case that $p\leq 2$, and thus $a$ must decay at infinity. In Section~\ref{S:SFAP2} we extend this to the case of $p>2$ where we encounter the challenges discussed above. In Section~\ref{S:Reduction} we construct a minimal compact solution $u$, under the assumption that Theorem~\ref{T} fails. Finally, in Section~\ref{S:Preclusion} we show that such compact solutions are incompatible with the Morawetz estimates of Nakanishi, completing the proof of Theorem~\ref{T}.

\subsection*{Acknowledgements} The author is very grateful to his advisor Jason Murphy for a careful reading of this manuscript as well as helpful comments and revisions concerning an earlier version of this manuscript. This work was supported in part by NSF grant DMS-2350225 (P.I. Jason Murphy). 

\section{Preliminaries}\label{S:prelim}
\subsection{Notation}
We write $A\lesssim B$ to denote that $A\leq CB$ for some $C>0$. We use the space-time Lebesgue norm
\[
\|u\|_{L_t^\alpha L_x^\beta(I\times \R^d)}=\|\|u(t,\cdot)\|_{L_x^\beta(\R^d)}\|_{L_t^\alpha(I)}.
\]
We assume familiarity with standard tools of harmonic analysis, such as Littlewood--Paley projections and associated estimates such as Bernstein estimates and the square function estimate. We use the standard notation $P_Nf$ or $f_N$ for the projection to frequency $N\in 2^{\mathbb{Z}}$. 

We denote the free Schr\"odinger propagator by $e^{it\Delta}$.  This is the Fourier multiplier operator with symbol $e^{-it|\xi|^2}$.  We recall the \emph{Duhamel formula} for solutions to $(i\partial_t + \Delta) u = F$:
\begin{equation}\label{Duham}
u(t)=e^{it\Delta}u_0-i\int_0^te^{i(t-s)\Delta}F(s)\,ds,\qtq{where} u_0=u|_{t=0}.
\end{equation}

\subsection{Basic Estimates}
We recall the Strichartz estimates for $e^{it\Delta}$.  We state the estimates in general space dimension $d$. 

A pair $(\alpha,\beta)$ is \textit{admissible} if $2\leq \alpha,\beta\leq \infty$, $\frac{2}{\alpha}+\frac{d}{\beta}=\frac{d}{2}$, and $(\alpha,\beta,d)\not=(2,\infty,2)$.

\begin{theorem}[Strichartz estimates \cite{GV, S2, KT}]
For any admissible pairs $(\alpha,\beta)$, $(\tilde \alpha,\tilde \beta)$ and $u:I\times\R^d\to\C$ a solution to $(i\partial_t+\Delta)u=F$ for some $F$, we have
\[
\|u\|_{L_t^\alpha L_x^\beta(I\times\R^d)}\lesssim \|u(t_0)\|_{L_x^2(\R^d)}+\|F\|_{L_t^{\tilde{\alpha}'}L_x^{\tilde{\beta}'}(I\times \R^d)}
\]
for any $t_0\in I$.  Here $'$ denotes the H\"older dual. 
\end{theorem}

We will also need the following standard pointwise estimates related to the nonlinearity.

\begin{lemma}
Let $f(z)=|z|^pz$ and let $c_1,c_2,...,c_J\in \C$. Then we have the following two pointwise estimates:
\begin{equation}\label{E:PW_1}
|f(c_1+c_2)-f(c_1)|\lesssim_p|c_2|^{p+1}+|c_2||c_1|^p,
\end{equation}
\begin{equation}\label{E:PW_2}
\left|f\left(\sum_{j=1}^Jc_j\right)-\sum_{j=1}^Jf(c_j)\right|\lesssim_{J,p}\sum_{j\not=k}|c_jc_k^p|.
\end{equation}
\end{lemma}

\subsection{Concentration Compactness}
We will need the following linear profile decomposition adapted to the Strichartz estimates which will be used in the Palais-Smale condition in Section~\ref{S:Reduction}. 
\begin{proposition}\label{P:LPD}
Let $2<r<\infty$, let $\frac{2r}{r-2}<q<\infty$, and let $\{f_n\}$ be a bounded sequence in $H^1$. Then there exists $J^*\in\{0,1,...,\infty\}$, nonzero profiles $\{\varphi^j\}_{j=1}^{J^*}\subset H^1$, and space-time sequences $\{(t_n^j,x_n^j)\}\subset \R\times\R^2$ such that the following decomposition holds for each finite $J\leq J^*$:
\[
f_n(x)=\sum_{j=1}^Je^{it_n^j\Delta}\varphi^j(x-x_n^j)+w_n^J,
\]
with $\{w_n^J\}$ bounded in $H^1$. Moreover, the following properties hold:
\begin{itemize}
\item The remainders vanish in the sense that
\[
\lim_{J\to J^*}\limsup_{n\to\infty}\|e^{it\Delta}w_n^J\|_{L_t^qL_x^r(\R)}=0.
\]
\item For each $j$, we have either $t_n^j\equiv0$ or $|t_n^j|\to\infty$ as $n\to\infty$. 

\item Similarly, for each $j$ either $x_n^j\equiv0$ or $|x_n^j|\to\infty$. In the latter case, writing  $x_n^j=(r_n^j\cos(\theta_n^j),r_n^j\sin(\theta_n^j))$, then the angles $\theta_n^j$ satisfy $\theta_n^j\to \theta_\infty^j$ as $n\to\infty$ for some $\theta_\infty^j\in \mathbb{S}^1$.

\item The profiles are asymptotically orthogonal in the sense that
\[
\lim_{n\to\infty}\{|t_n^j-t_n^k|+|x_n^j-x_n^k|\}=\infty.
\]

\item For each $J$ we have the following mass/energy decoupling:
\[
M(f_n)=\sum_{j=1}^JM(\varphi^j)+M(w_n^J)+o_n(1)\qtq{as}n\to\infty,
\]
\[
E(f_n)=\sum_{j=1}^JE(e^{it_n^j\Delta}\varphi^j)+E(w_n^J)+o_n(1)\qtq{as}n\to\infty.
\]
\end{itemize}
\end{proposition}

The proof is standard and there are many great resources treating the proof in depth (cf. \cite{Keraani,FXC,KV,V}). For this reason we will only sketch the proof below. There are two key differences in the version above from other variants, namely the choice of scattering norm and the condition concerning the angles $\theta_n$, neither of which significantly impacts the argument.

We now sketch the proof. The idea is to identify ``bubbles of concentration" and show that after removing these bubbles the error vanishes in the desired norm. To identify such bubbles one must identify a physical scale and a point in space-time where concentration occurs. The key to finding such a scale is following lemma:

\begin{lemma}[Refined Sobolev Embedding]\label{L:RSE}
Let $2<r<\infty$. Then there exists some $\sigma=\sigma(r)$ such that for any $f\in H^1$ one has
\[
\|f\|_{L^r(\R^2)}\lesssim\sup_{N\in 2^\Z}\|P_Nf\|_{L^r(\R^2)}^\sigma\|f\|_{H^1(\R^2)}^{1-\sigma}.
\]
\end{lemma}
\begin{proof}
For $2<r\leq 4$ one may estimate using the square-function estimate, concavity of fractional powers, H\"older's inequality, and Bernstein estimates
\begin{align*}
\|f\|_{L^r}^r&\lesssim\sum_{N_1\leq N_2}\int |f_{N_1}|^\frac{r}{2}|f_{N_2}|^\frac{r}{2}\,dx\\
&\lesssim\sum_{N_1\leq N_2}\|f_{N_1}\|_{L^\frac{2r}{4-r}}\|f_{N_1}\|_{L^r}^\frac{r-2}{2}\|f_{N_2}\|_{L^r}^{\frac{r-2}{2}}\|f_{N_1}\|_{L^2}\\
&\lesssim\sup_{N}\|f_N\|_{L^r}^{r-2}\sum_{N_1\leq N_2}(N_1N_2)^\frac{r-2}{r}\||\nabla|^{\frac{r-2}{r}}f_{N_1}\|_{L^\frac{2r}{4-r}}\||\nabla|^\frac{r-2}{r}f_{N_2}\|_{L^2}\\
&\lesssim \sup_N\|f_N\|_{L^r}^{r-2}\sum_{N_1\leq N_2}\left(\frac{N_1}{N_2}\right)^\frac{r-2}{r}\||\nabla|^\frac{r-2}{r}f_{N_1}\|_{L^2}\||\nabla|^\frac{r-2}{r}f_{N_1}\|_{L^2}\\
&\lesssim \sup_N\|f_{N}\|_{L^r}^{r-2}\||\nabla|^\frac{r-2}{r}f\|_{L^2}.
\end{align*}
For $4<r<\infty$ one should first use the Gagliardo-Nirenberg inequality to get
\[
\|f\|_{L^r}\lesssim \|f\|_{L^4}^{\frac{4}{r}}\|f\|_{\dot{H}^1}^{\frac{r-4}{r}},
\]
and then estimate the $\|f\|_{L^4}$ norm as above.
\end{proof}

Having built a tool to identify a scale of concentration, we now show how to identify a bubble of concentration:

\begin{lemma}[Inverse Strichartz Inequality]\label{L:ISE}
Let $2<r<\infty$, let $\frac{2r}{r-2}<q<\infty$, and let $\{f_n\}$ be a sequence in $H^1$ satisfying
\[
\lim_{n\to\infty}\|f_n\|_{H^1}=A\qtq{and}\lim_{n\to\infty}\|e^{it\Delta}f_n\|_{L_t^qL_x^r}=\varepsilon>0.
\]
Passing to a subsequence, there exist $\phi\in H^1$ with
\begin{equation}\label{E:ISE1}
\|\phi\|_{H^1}\gtrsim A(\tfrac{\varepsilon}{A})^\alpha
\end{equation}
for some $\alpha=\alpha(q,r)\in (0,1)$, as well as $(t_n,x_n)\in\R\times\R$ such that
\begin{align}
&e^{it_n\Delta}f_n(x+x_n)\rightharpoonup \phi(x)\qtq{weakly in } H^1, \label{E:ISE2} \\
&\liminf_{n\to\infty}\left\{\|f_n\|_{H^\lambda}^2-\|f_n-\phi_n\|_{H^\lambda}^2-\|\phi\|_{H^\lambda}^2\right\}=0,\label{E:ISE3}
\end{align}
for $\lambda\in\{0,1\}$. Additionally, for any $p>0$ one has 
\begin{equation}\label{E:ISE4}
\liminf_{n\to\infty}\left\{\|a\,f_n\|_{L_x^{p+2}}^{p+2}-\|a\,[f_n-\phi_n]\|_{L_x^{p+2}}^{p+2}-\|a\,\phi_n\|_{L_x^{p+2}}^{p+2}\right\}=0,
\end{equation}
where $\phi_n = e^{it_n\Delta}\phi(x-x_n)$. Finally, we may also take either $t_n\equiv0$ or $|t_n|\to\infty$. Likewise, we may take either $|x_n|\equiv0$ or $|x_n|\to\infty$, and in this latter case writing $x_n=(r_n\cos(\theta_n),r_n\sin(\theta_n))$ we have that $\theta_n\to \theta_\infty$ for some $\theta_\infty\in \mathbb{S}^1$.
\end{lemma}
\begin{proof}
Passing to a subsequence, we may assume that for all $n$ we have
\[
\|f_n\|_{H^1}\sim A\qtq{and}\|e^{it\Delta}f_n\|_{L_t^qL_x^r}\sim \varepsilon.
\]
Then we have by H\"older's inequality, Strichartz estimates, the refined Sobolev embedding Lemma~\ref{L:RSE}
\begin{align*}
\|e^{it\Delta}f_n\|_{L_t^qL_x^r}&\lesssim \|e^{it\Delta}f_n\|_{L_t^\infty L_x^r}^{\frac{q(r-2)-2r}{q(r-2)}}\|e^{it\Delta}f_n\|_{L_t^\frac{2r}{r-2}L_x^r}^{\frac{2r}{q(r-2)}}\\
&\lesssim \|e^{it\Delta}f_n\|_{L_t^\infty L_x^r}^{\frac{q(r-2)-2r}{q(r-2)}}\|f_n\|_{H^1}^{\frac{2r}{q(r-2)}}.
\end{align*}
Thus for $\sigma_1=\frac{q(r-2)-2r}{q(r-2)}$ we may find a sequence of times $\{t_n\}\subset \R$ such that
\[
A\left(\frac{\varepsilon}{A}\right)^\frac{1}{\sigma_1}\lesssim \|e^{it_n\Delta}f_n\|_{L_t^r}.
\]
Next, by the Refined Sobolev Embedding (Lemma~\ref{L:RSE}) we may find some $\sigma_2$ and frequency scales $\{N_n\}\subset 2^\Z$ such that
\[
A\left(\frac{\varepsilon}{A}\right)^\frac{1}{\sigma_1\sigma_2}\lesssim \|e^{it_n\Delta}P_{N_n}f_n\|_{L^r}.
\]
By Bernstein estimates we have
\[
A\left(\frac{\varepsilon}{A}\right)^\frac{1}{\sigma_1\sigma_2}\lesssim N_{n}^\frac{r-2}{r}\|f_n\|_{L^2}\qtq{so that} \left(\frac{\varepsilon}{A}\right)^{\frac{1}{\sigma_1\sigma_2}}\lesssim N_n^\frac{r-2}{r}.
\]
Similarly, we have
\[
A\left(\frac{\varepsilon}{A}\right)^\frac{1}{\sigma_1\sigma_2}\lesssim N_n^{-\frac{2}{r}}\||\nabla|^1f_n\|_{L^2}\qtq{so that} \left(\frac{\varepsilon}{A}\right)^{\frac{1}{\sigma_1\sigma_2}}\lesssim N_n^{-\frac{2}{r}}.
\]
Together these inequalities provide an upper and lower bound for $N_n$ in terms of $A$ and $\varepsilon$, thus after passing to a further subsequence we may assume that $N_n\equiv M$ for some $M$.

Finally, using H\"older's inequality with $\sigma_3=\frac{2}{r}$ we have
\[
\|e^{it_n\Delta}P_{N_n}f_n\|_{L^r}\leq \|e^{it_n\Delta}P_{N_n}f_n\|_{L^2}^{1-\sigma_3}\|e^{it_n\Delta}P_{N_n}f_n\|_{L^\infty}^{\sigma_3},
\]
so that putting $\sigma=\sigma_1\sigma_2\sigma_3$ there are a sequence of translation parameters $\{x_n\}\subset \R^2$ such that
\[
A\left(\frac{\varepsilon}{A}\right)^\frac{1}{\sigma}\lesssim |[e^{it_n\Delta}P_{N_n}f_n](x_n)|.
\]
Now we define $g_n(x)=e^{it_n\Delta}f_n(x+x_n)$. We observe that for all $n$ $\|g_n\|_{H^1}=\|f_n\|_{H^1}\lesssim A$, thus passing to a further subsequence there is some $\phi\in H^1$ such that the $g_n\rightharpoonup \phi$ weakly in $H^1$, giving \eqref{E:ISE2} and \eqref{E:ISE3}. We can deduce \eqref{E:ISE1} by pairing $\phi$ with the convolution kernel for $P_M$ and we may deduce the potential energy decoupling, \eqref{E:ISE4}, via an argument involving the refined Fatou's Lemma of Br\'ezis and Lieb (cf. \cite{FXC, KV, V}).

Finally, for the behavior of the $t_n$ and $x_n$ by passing to further subsequences we may assume that either $t_n\to t_\infty\in \R$ or $|t_n|\to\infty$. We may then in the first case, by incorporating the translation into the definition of the profile, take $t_n\equiv0$. The same logic applies to the $x_n$ and we obtain the statement concerning the angles $\theta_n$ by using compactness of $\mathbb{S}^1$ and passing to a further subsequence.
\end{proof}

\begin{proof}[Sketch of the proof of Proposition~\ref{P:LPD}]
We set $w_n^0=f_n$. If $\|e^{it\Delta}w_n^0\|_{L_t^qL_x^r}\to 0$, then we may set $J^*=0$ and we are done. Otherwise, we may apply the preceding Lemma~\ref{L:ISE} to obtain paramters $(t_n^1,x_n^1)$ and a profile $\varphi^1$. Setting $w_n^1=w_n^0-e^{-it_n^1\Delta}\varphi^1(x-x_n)$, we again have that if $\|e^{it\Delta}w_n^1\|_{L_t^qL_x^r}\to 0$, then we set $J^*=1$ and again we are done. Otherwise, we repeat this process until the remainder vanishes at some finite $J^*$ or we set $J^*=\infty$. In the latter case one may still confirm that the remainder vanishes as $J\to \infty$.

The $H^1$ and potential energy decoupling from Lemma~\ref{L:ISE} will provide the mass and energy decoupling statements.

Similarly, the weak convergence properties will imply the asymptotic orthogonality of the parameters $(t_n^j,x_n^j)$. Once again for full details we refer the reader to \cite{Keraani, FXC, KV, V}.
\end{proof}

\subsection{Nonlinear Estimates}

We now introduce the function spaces which we will use and the related nonlinear estimates. For a time interval $I$ we define
\[
X(I)=L_t^qL_x^r(I\times\R^2)\qtq{and} Y(I)=L_t^{\alpha}L_x^\beta(I\times\R^2),
\]
where
\begin{equation}\label{E:ExpDefsqr}
q=q(p)=\begin{cases}
\frac{p^2+2p}{2}&p>2\\
2p+4 & p\leq2,
\end{cases}\quad r=r(p)=\begin{cases}
p+2 & p>2\\
\frac{4p+8}{p} & p\leq 2,
\end{cases}
\end{equation}

\begin{equation}\label{E:ExpDefab}
\alpha=\alpha(p)=\begin{cases}
\frac{p^2+2p}{2p+2} & p>2\\
\frac{2p+4}{p+1} & p\leq 2,
\end{cases}
\qtq{and}
\beta=\beta(p)=\begin{cases}
\frac{p+2}{p+1} & p>2,\\
\frac{4p+8}{3p+8} & p\leq 2.
\end{cases}
\end{equation}

Throughout the rest of the paper $q,r,\alpha$, and $\beta$ will always refer to the above definitions. We also recall the values of $\rho$ \eqref{E:RhoValue} and $s(p)$ \eqref{E:ScalValue}. We note that the space $X(I)$ is ``critical" in the sense that is scales like $L_t^\infty \dot{H}_x^{s(p)}$ (cf. \eqref{E:ScalValue}).

Additionally, for any choice of $p>0$ one has that
\[
\frac{1}{q}+\frac{1}{r}<\frac{1}{2}
\]
which implies $q>\frac{2r}{r-2}$ and $r>\frac{2q}{q-2}$. In particular Proposition \ref{P:LPD} holds with the remainder vanishing in $X(\R)$, and by Sobolev embedding one has
\begin{equation}\label{E:SobforX}
\|u\|_{X(I)}\lesssim \||\nabla|^{s}u\|_{L_t^qL_x^{\frac{2q}{q-2}}(I\times\R^2)}.
\end{equation}

The spaces $X(I)$ and $Y(I)$ are related by the following exotic Strichartz estimate. It can be proven in the standard fashion, i.e. by applying the dispersive estimates for $e^{it\Delta}$ and using the Hardy-Littlewood-Sobolev inequality.

\begin{lemma}
Let $I$ be a compact time interval containing $t_0$. Then for all $t\in I$ and $g\in Y(I)$ we have
\begin{equation}\label{E:ExoStric}
\left\|\int_{t_0}^te^{i(t-s)\Delta}g(s)\,ds\right\|_{X(I)}\lesssim \|g(s)\|_{Y(I)}.
\end{equation}
\end{lemma}
Finally, an application of H\"older's inequality also yields the following nonlinear estimate
\begin{lemma}
Let $g,h:I\times\R^2\to \C$ and let $a:\R^2\to \R$ be admissible as in Defintion~\ref{D:ADef}. Then 
\begin{equation}\label{E:NonLinEst}
\|agh^p\|_{Y(I)}\leq \|a\|_{L_x^\rho(\R^2)}\|g\|_{X(I)}\|h\|_{X(I)}^p.
\end{equation}
\end{lemma}

\section{Well-Posedness and Stability}\label{S:Stability}

Using standard arguments we can obtain local well-posedness for solutions to \eqref{nls} under the assumption $a,\nabla a\in L^\infty$. Using conservation of mass and energy we can then obtain global well-posedness. Specifically, we obtain the following result.

\begin{proposition}[Well-Posedness]\label{P:WellPosed}
Let $p>0$ and fix an admissible $a$ as in Definition~\ref{D:ADef}. Then we have
\begin{itemize}
\item[(i)]\textbf{Initial-value problem.} For any $u_0\in H^1(\R^2)$, there exists a unique global solution to \eqref{nls} with $u|_{t=0}=u_0$. Further, we have
\[
\|u\|_{L_t^\infty H_x^1(\R\times\R^2)}\lesssim \|u_0\|_{H^1(\R^2)}.
\]
\item[(ii)]\textbf{Final-state problem.} For any $u_+\in H^1(\R^2)$, there exists a unique global solution $u$ to \eqref{nls} such that
\[
\lim_{t\to\infty}\|u(t)-e^{it\Delta}u_+\|_{H^1}=0,
\]
and further we have
\[
\|u\|_{L_t^\infty H_x^1(\R\times\R^2)}\lesssim \|u_0\|_{H^1(\R^2)}.
\]
A similar statement holds backward in time.
\item[(iii)]\textbf{Scattering criterion.} If $u$ is a solution to \eqref{nls} as above satisfying
\[
\|u\|_{X(\R)}<\infty,
\]
then $u$ scatters in $H^1$ both forward and backward in time.
\item[(iv)]\textbf{Small-data theory.} If $u$ is a solution to \eqref{nls} as above with $M(u)+E(u)$ sufficiently small, then 
\[
\|u\|_{X(\R)}\lesssim \|u\|_{L_t^\infty H_x^1(\R\times\R^2)}.
\]
In particular, $u$ scatters in $H^1$ both forward and backward in time.
\end{itemize}
\end{proposition}

A crucial part of establishing Theorem~\ref{T} is a suitable stability theory for \eqref{nls}. However, as mentioned in Section~\ref{S:Intro}, to develop such a theory for all $p>0$ introduces new obstacles. Here we follow the approach of \cite{HR, M-half} and work with spaces which scale critically but which do not involve derivatives. The advantage of such spaces is that in what follows the error is only required to be small in a space without derivatives. In fact, because the problem is always $H^1$ subcritical and global existence is already guaranteed for $H^1$ data, we don't need to control any derivatives of the error.

We first establish our result on time intervals for which the $X(I)$-norm of the approximate solution is assumed small. 

In what follows we will use the notation $f(z)=|z|^pz$ for some $p>0$ and $F(x,z)=a(x)f(z)$, for an admissible $a$. When there is no risk of confusion we will suppress this $x$-dependence.

\begin{proposition}[Short-time perturbations] Let $t_0\in I$, $u_0\in H^1$, and suppose $\tilde{u}:I\times\R\to\C$ is a solution to the equation
\[
(i\partial_t+\Delta)\tilde{u}=F(\tilde{u})+e,
\]
for some function $e$. Then there exist $\varepsilon_0=\varepsilon_0(\|a\|_{L^\rho})$ and $\delta=\delta(\|a\|_{L^\rho})$ such that for all $0<\varepsilon\leq \varepsilon_0$ if
\begin{equation}\label{E:STP_A1}
\|\tilde{u}\|_{X(I)}<\delta,
\end{equation}
\begin{equation}\label{E:STP_A2}
\|e^{i(t-t_0)\Delta}(u_0-\tilde{u}_0)\|_{X(I)}<\varepsilon,
\end{equation}
and
\begin{equation}\label{E:STP_A3}
\|e\|_{Y(I)}<\varepsilon,
\end{equation}
then there is a solution to \eqref{nls}, $u$, with $u(t_0)=u_0$ satisfying
\begin{equation}\label{E:STP_C1}
\|u-\tilde{u}\|_{X(I)}\lesssim \varepsilon,
\end{equation}
and
\begin{equation}\label{E:STP_C2}
\|F(u)-F(\tilde{u})\|_{Y(I)}\lesssim \varepsilon.
\end{equation}
\end{proposition}
\begin{proof}
We have that $u_0\in H^1$ guarantees  $u$ exists and is global. Setting $w=u-\tilde{u}$, we observe that $w$ solves the equation
\[
(i\partial_t+\Delta)w=F(\tilde{u}+w)-F(\tilde{u})-e.
\]
Now using the Duhamel formulation we may estimate via the exotic Strichartz estimate \eqref{E:ExoStric}, \eqref{E:STP_A2}, \eqref{E:STP_A3}, the pointwise estimate for the nonlinearity \eqref{E:PW_1}, the nonlinear estimate \eqref{E:NonLinEst}, and \eqref{E:STP_A3} again,
\begin{align*}
\|w\|_{X(I)} &\lesssim \varepsilon+\|F(\tilde{u}+w)-F(\tilde{u})\|_{Y(I)}+\|e\|_{Y(I)}\\
&\lesssim \varepsilon+ \|a\|_{L_x^\rho}(\|w\|_{X(I)}^{p+1}+\|w\|_{X(I)}\|\tilde{u}\|_{X(I)}^p)\\
&\lesssim_a \varepsilon+\delta^p\|w\|_{X(I)}+\|w\|_{X(I)}^{p+1}.
\end{align*}
In particular, for $\delta$ sufficiently small  we obtain
\[
\|w\|_{X(I)}\lesssim_a \varepsilon+\|w\|_{X(I)}^{p+1},
\]
so that for $\varepsilon$ sufficiently small we obtain \eqref{E:STP_C1}. To conclude \eqref{E:STP_C2} we estimate as above, using \eqref{E:PW_1}, \eqref{E:NonLinEst}, \eqref{E:STP_A1}, and \eqref{E:STP_C1}.
\end{proof}

By iterating the above proposition we obtain our stability result, which removes the smallness assumption on the $X(I)$-norm of the approximate solution. 

\begin{proposition}[Stability]\label{P:Stability}
Let $t_0\in I$, $u_0\in H^1$, and suppose $\tilde{u}:I\times\R^2\to \C$ is a solution to the equation
\[
(i\partial_t+\Delta)\tilde{u}=F(\tilde{u})+e,
\]
satisfying
\begin{equation}\label{E:Stab_A1}
\|\tilde{u}\|_{X(I)}\leq L,
\end{equation}
for some $L>0$. Additionally suppose
\begin{equation}\label{E:Stab_A2}
\|u_0-\tilde{u}_0\|_{H^1}+\|e\|_{Y(I)}<\varepsilon,
\end{equation}
for some $\varepsilon\leq \varepsilon_1=\varepsilon_1(L)$. Then the global solution $u$ to \eqref{nls} with initial data $u(t_0)=u_0$ satisfies
\[
\|u-\tilde{u}\|_{X(I)}+ \|F(u)-F(\tilde{u})\|_{Y(I)}\lesssim\varepsilon.
\]
\end{proposition}

We omit the proof as it is standard (see \cite{KV} for an in-depth treatment), however we remark that there is a mismatch in the smallness condition placed on $u_0-\tilde{u}_0$ here compared to in the short-time case. The point is that this stronger requirement paired with global existence of $u$ allow us obtain the weaker condition upon each iteration of the short-time case.

\section{Scattering Far Away Profiles: $p\leq 2$}\label{S:SFAP1}

As mentioned in Section~\ref{S:Intro} a key ingredient in establishing Theorem~\ref{T} is to show that if Theorem~\ref{T} were to fail, we could construct a solution to \eqref{nls} living at the mass/energy threshold whose orbit in $H^1$ is pre-compact. This is carried out in Section~\ref{S:Reduction}, with the majority of the work contained in the Palais-Smale condition Proposition~\ref{P:PSmale}. The idea will be to construct nonlinear profiles corresponding to the linear profiles obtained by applying the linear profile decomposition, Proposition~\ref{P:LPD}. Due to the broken translation symmetry, in the case that the $|x_n^j|\to\infty$ we cannot simply apply a translation to construct the desired solution. In the case that $p\leq2$, due to the assumption that $a$ decays at infinity, we expect the nonlinearity $a|u|^pu$ to be very weak far from the origin. Hence, we expect that a suitable solution to the free problem may serve as a good approximate solution to \eqref{nls}. The following proposition makes this precise.

\begin{proposition}[Scattering for far-away profiles: $p\leq 2$]\label{P:SFAP1}
Let $0<p\leq 2$, let $a$ be admissible in the sense of Definition~\ref{D:ADef}, and let $\varphi\in H^1$. Let $\{t_n\}\subseteq \R$ be such that $t_n\equiv0$ or $|t_n|\to\infty$ and let $\{x_n\}\subset \R^2$ satisfy $|x_n|\to \infty$. Then for $n$ sufficiently large there exists a global solution, $v_n$, to \eqref{nls} with
\[
v_n(0,x)=\varphi_n(x)=e^{it_n\Delta}\varphi(x-x_n)
\]
that obeys the bound
\[
\|v_n\|_{X(\R)}\lesssim_{\|\varphi\|_{H^1}}1.
\]
In particular, $v_n$ scatters both forward and backward in time. Furthermore, for any $\varepsilon>0$ there exists $N$ and $\psi\in C_c^\infty(\R\times\R^2)$ such that for $n\geq N$,
\[
\|v_n-\psi(\cdot+t_n,\cdot-x_n)\|_{X(\R)}<\varepsilon.
\]
\end{proposition}
\begin{proof}
Note that for $p\leq 2$, the repulsivity condition on $a$ paired with the fact that $a\in L^1$ implies that $\displaystyle \lim_{|x|\to \infty}a(x)=0$.

For each $n$ we define smooth cutoff functions satisfying
\[
\chi_n(x)=\begin{cases}
    1 & |x+x_n|>\frac{1}{2}|x_n|\\
    0 & |x+x_n|<\frac{1}{4}|x_n|,
\end{cases}
\]
that obey the symbol bounds $|\partial^{\kappa}\chi_n|\lesssim |x_n|^{-|\kappa|}$, for multiindices $\kappa$. We remark that $\chi_n\to 1$ pointwise as $n\to\infty$.

We now define for $T\geq 1$ define approximate solutions $\tilde{v}_{n,T}$ as follows: For $|t|\leq T$ we define
\[
\tilde v_{n,T}(t,x)=\chi_n(x-x_n)e^{it\Delta}\varphi(x-x_n),
\]
and for $|t|>T$ we define $\tilde{v}_{n,T}$ by evolving under the free equation, so that
\[
\tilde v_{n,T}(t,x)=\begin{cases}
    e^{i(t-T)\Delta}\tilde v_{n,T}(T) & t>T\\
    e^{i(t+T)\Delta}\tilde{v}_{n,T}(-T) & t<T.
\end{cases}
\]
Our goal will be to show that for $n$ and $T$ sufficiently large, the $\tilde{v}_{n,T}$ define suitable approximate solutions to \eqref{nls}, and then to apply stability to construct our true solution. To this end we will verify the following three claims conerning the $\tilde{v}_{n,T}$:
\begin{equation}\label{E:SFAP1_C1}
\lim_{T\to\infty}\limsup_{n\to\infty}\|\tilde{v}_{n,T}(t_n)-\varphi_n\|_{H^1}=0,
\end{equation}
\begin{equation}\label{E:SFAP1_C2}
\limsup_{T\to\infty}\limsup_{n\to\infty}\|\tilde{v}_{n,T}\|_{X(\R)}\lesssim_{\|\varphi\|_{H^1}}1,
\end{equation}
and
\begin{equation}\label{E:SFAP1_C3}
\lim_{T\to\infty}\limsup_{n\to\infty}\|e_{n,T}\|_{Y(\R)}=0,
\end{equation}
where
\[
e_{n,T}:=(i\partial_t+\Delta)\tilde{v}_{n,T}-a(x)|\tilde{v}_{n,T}|^p\tilde v_{n,T}.
\]
\begin{proof}[Proof of \eqref{E:SFAP1_C1}]
First suppose that $t_n\equiv0$, so that
\begin{align*}
\|\tilde{v}_{n,T}(t_n)-\varphi_n\|_{H^1}=\|(\chi_n-1)\varphi\|_{H^1},
\end{align*}
which tends to $0$ as $n\to\infty$ by the Dominated Convergence Theorem. If instead $t_n\to \infty$, then choosing $n$ sufficiently large we have $t_n>T$, so that
\begin{align*}
\|\tilde{v}_{n,T}(t_n)-\varphi_n\|_{H^1}&=\|(\chi_n-1)e^{iT\Delta}\varphi\|_{H^1},
\end{align*}
which again tends to $0$ as $n\to\infty$ by the Dominated Convergence Theorem. The case of $t_n\to-\infty$ is similar.
\end{proof}

\begin{proof}[Proof of \eqref{E:SFAP1_C2}]
We estimate on the regions $|t|\leq T$ and $|t|>T$ separately. For $|t|\leq T$, we estimate by H\"older's inequality, \eqref{E:SobforX}, and Strichartz estimates 
\begin{align*}
\|\tilde{v}_{n,T}\|_{X([-T,T])}&\lesssim \|\chi_n\|_{L^\infty}\|e^{it\Delta}|\nabla|^{s}\varphi\|_{L_t^qL_x^\frac{2q}{q-2}([-T,T]\times\R^2)}\\
&\lesssim \|\varphi\|_{H^1}.
\end{align*}
On the region $t>T$ we estimate by \eqref{E:SobforX} and Strichartz estimates
\begin{align*}
\|\tilde{v}_{n,T}\|_{X(T,\infty)}&\lesssim\|e^{i(t-T)\Delta}|\nabla|^{s} \tilde{v}_{n,T}(T)\|_{L_t^qL_x^\frac{2q}{q-2}}\\
&\lesssim\|\tilde{v}_{n,T}(T)\|_{H_x^1},
\end{align*}
and likewise for $t<-T$ we will have 
\[
\|\tilde{v}_{n,T}\|_{X(-\infty,-T)}\lesssim \|\tilde{v}_{n,T}(-T)\|_{H_x^1}.
\]
Thus the claim will follow once we establish uniform $L_t^\infty H_x^1([-T,T]\times\R^2)$-bounds for the $\tilde{v}_{n,T}$. The $L^2$-bounds are immediate, and for the derivative we estimate by the product rule and H\"older's inequality:
\begin{align*}
\|\nabla \tilde{v}_{n,T}\|_{L_t^\infty L_x^2}&\lesssim \|\nabla \chi_n\|_{L^\infty_x}\|e^{it\Delta}\varphi\|_{L_t^\infty L_x^2}+\|\chi_n\|_{L_x^\infty}\|e^{it\Delta}\nabla \varphi\|_{L_t^\infty L_x^2}\\
&\lesssim \|\varphi\|_{H^1},
\end{align*}
which is acceptable.
\end{proof}
\begin{proof}[Proof of \eqref{E:SFAP1_C3}]
Again we consider the temporal regions separately. For $|t|>T$ we have
\[
e_{n,T}=-a(x)|\tilde{v}_{n,T}|^p\tilde{v}_{n,T}.
\]
Thus, we may estimate via the nonlinear estimate \eqref{E:NonLinEst}, 
\begin{align*}
\|e_{n,T}\|_{Y(T,\infty)}&\lesssim \|a\|_{L_x^\rho}\|\tilde{v}_{n,T}\|_{X(T,\infty)}^{p+1},
\end{align*}
so that by \eqref{E:SFAP1_C2} the above tends to $0$ as $n,T\to\infty$ by the Dominated Convergence Theorem. The estimate on $(-\infty,-T)$ is identical.

We are left with estimating the error on $[-T,T]$. Here we have
\begin{align}
e_{n,T}(t,x)&=\Delta \chi_n(x-x_n) e^{it\Delta}\varphi(x-x_n)\label{E:SFAP1_error1}\\
&\quad+2\nabla \chi_n(x-x_n)\cdot \nabla e^{it\Delta}\varphi(x-x_n)\label{E:SFAP1_error2}\\
&\quad-a(x)|\tilde{v}_{n,T}|^p\tilde{v}_{n,T}.\label{E:SFAP1_error3}
\end{align}
We now estimate the contributions of each of these terms separately. Recall the definition of the exponents $\alpha,\beta$ from \eqref{E:ExpDefab}. Now on the support of $\Delta\chi_n(x-x_n)$ we estimate \eqref{E:SFAP1_C1} by H\"older's inequality, the symbol bounds on the $\chi_n$, and Strichartz estimates:
\begin{align*}
\|\eqref{E:SFAP1_error1}\|_{Y([-T,T])}&\lesssim |T|^{\frac{1}{\alpha}}\|\Delta\chi_n\|_{L^\frac{2\beta}{2-\beta}}\|e^{it\Delta}\varphi\|_{L_t^\infty L_x^2}\\
&\lesssim |T|^\frac{1}{\alpha}|x_n|^{-\frac{2+\beta}{\beta}}\|\varphi\|_{L^2}\\
&\to 0\qtq{as} n\to\infty.
\end{align*}
Similarly, for \eqref{E:SFAP1_error2} we estimate on the support of $\nabla\chi_n(x-x_n)$ again by H\"older's inequality, the symbol bounds, and Strichartz estimates:
\begin{align*}
\|\eqref{E:SFAP1_error2}\|_{Y([-T,T])}&\lesssim |T|^\frac{1}{\alpha }\|\nabla\chi_n\|_{L^\frac{2\beta}{2-\beta}}\|e^{it\Delta}\nabla \varphi\|_{L_t^\infty L_x^2}\\
&\lesssim |T|^\frac{1}{\alpha}|x_n|^{-\frac{2}{\beta}}\|\varphi\|_{H^1}\\
&\to 0\qtq{as}n\to\infty.
\end{align*}
Finally, for \eqref{E:SFAP1_error3} we estimate via the nonlinear estimate \eqref{E:NonLinEst}
\begin{align*}
\|\eqref{E:SFAP1_error3}\|_{Y([-T,T])}&\lesssim \|a\chi_n(\cdot-x_n)\|_{L_x^\rho}\|\tilde{v}_{n,T}\|^{p+1}_{X([-T,T])}\\
&\lesssim \|a\|_{L_x^{\rho}(|x|>\frac{1}{4}|x_n|)}\|\tilde{v}_{n,T}\|_{X([-T,T])}^{p+1},
\end{align*}
which tends to $0$ as $n\to\infty$ by the fact that $\rho<\infty$, \eqref{E:SFAP1_C2}, and the Dominated Convergence Theorem.
\end{proof}

Having established \eqref{E:SFAP1_C1}--\eqref{E:SFAP1_C3}, we may now for sufficiently large $n$ and $T$ apply Proposition~\ref{P:Stability}. This gives true solutions $v_n$ such that after a time translation $v_n(0,x)=\varphi_n(x)$.

It remains to establish approximation by $C_c^\infty(\R\times\R^2)$ functions. Given $\varepsilon>0$ we may find $\psi\in C_c^\infty(\R\times\R^2)$ such that
\[
\|\psi-e^{it\Delta}\varphi\|_{X(\R)}<\frac{\varepsilon}{2}.
\]
By construction, the $v_n$ satisfy
\[
\lim_{T\to\infty}\limsup_{n\to\infty}\|v_n(\cdot-t_n)-\tilde{v}_{n,T}(\cdot)\|_{X(\R)}= 0,
\]
so that by the triangle inequality it suffices to show that
\[
\lim_{T\to\infty}\limsup_{n\to\infty}\|\tilde{v}_{n,T}-e^{it\Delta}\varphi(x-x_n)\|_{X(\R)}=0.
\]
On the region $|t|<T$ we estimate
\[
\|\tilde{v}_{n,T}-e^{it\Delta}\varphi(x-x_n)\|_{X([-T,T])}=\|(\chi_n-1)e^{it\Delta}\varphi\|_{X([-T,T])}\to 0\qtq{as}n\to\infty,
\]
by the Dominated Convergence Theorem. For $t>T$ we instead estimate by \eqref{E:SobforX} on $(T,\infty)$
\begin{align*}
\|\tilde{v}_{n,T}-e^{it\Delta}\varphi(x-x_n)\|_{X}&\lesssim \||\nabla|^{s(p)}[\tilde{v}_{n,T}(T)-e^{iT\Delta}\varphi(x-x_n)]\|_{L_t^qL_x^\frac{2q}{q-2}}\\
&\lesssim \|(\chi_n-1)e^{iT\Delta}\varphi\|_{H^1}\to 0\qtq{as}n\to\infty.
\end{align*}
Finally, the region $t<-T$ is similar, establishing the result.
\end{proof}

\section{Scattering for Far Away Profiles: $p>2$}\label{S:SFAP2}

In this section we extend Proposition~\ref{P:SFAP1} to the case for $p>2$. In this regime $a$ need not decay at infinity. In fact depending on which direction one travels away from the origin, $a$ could take on different values. Thus the behavior of solutions living far from the origin is not so clear. By enforcing the map $\tilde{a}:\mathbb{S}^1\to \R$ defined by \eqref{E:ATilde} be continuous, we guarantee that traveling in similar directions gives similar limiting behavior of $a$, which partially resolves this issue. However, it is still not immediately clear how one can effectively localize around some fixed direction. Indeed, the resolution of this issue is the chief novelty of this paper. 

The approach taken here is to cut off to a family of triangular regions. By translating and shrinking the angular widths of these regions in a suitable way, we are able to ensure these cutoffs capture the behavior of $a$ in a single direction. We will also ensure these function have nice pointwise convergence properties and that their derivatives have nice support properties. After constructing this family of cutoff functions, the proof follows as in Proposition~\ref{P:SFAP1}.

The following lemma carries out the details of the construction.

\begin{lemma}\label{L:Cutoffs}
Let $p>2$ and let $a$ be admissible as in Definition~\ref{D:ADef} and recall the definition of $\tilde{a}$, \eqref{E:ATilde}. Let $\{(r_n\cos(\theta_n),r_n\sin(\theta_n))\}\subset \R^2$ satisfy $r_n\to \infty$, $\theta_n\to \theta_\infty$. Then there exist smooth functions $\chi_n$ such that for $n$ sufficiently large the following three properties hold:
\begin{equation}\label{E:SFAP2L_C1}
\chi_n\to 1\qtq{pointwise as}n\to\infty,
\end{equation}
\begin{equation}\label{E:SFAP2L_C2}
\|a(x)-\tilde{a}(\theta_\infty)\|_{L^\infty(\supp(\chi_n(\cdot-x_n)))}\to 0\qtq{as}n\to\infty,
\end{equation}
\begin{equation}\label{E:SFAP2L_C3}
\|\nabla\chi_n\|_{L^\frac{2\beta}{2-\beta}(\R^2)}\lesssim r_n^{-\frac{1}{p+2}}\qtq{and}\|\Delta\chi_n\|_{L^\frac{2\beta}{2-\beta}(\R^2)}\lesssim r_n^{-1},
\end{equation}
where $\beta$ is as in \eqref{E:ExpDefab}.
\end{lemma}
\begin{proof}
By rotating our coordinate system about $\theta_\infty$ we may assume that $\theta_\infty=0$. Now, we define a sequence of auxiliary angles $\omega_n\in \mathbb{S}^1$ such that

\begin{equation}\label{E:SFAP2L_P1}
\omega_n\to0
\end{equation}
and for $n$ sufficiently large one has
\begin{equation}\label{E:SFAP2L_P2}
\quad \sin(\omega_n)\geq r_n^{-\frac{1}{p+2}}
\end{equation}
and
\begin{equation}\label{E:SFAP2L_P3}
\left|\frac{\sin(\theta_n)}{\sin(\omega_n)}\right|\leq \frac{1}{16}.
\end{equation}
For example, we could take 
\[
\omega_n=\max\left(17|\theta_n|,\sin^{-1}\Big(r_n^{-\frac{1}{p+2}}\Big)\right).
\]
We introduce this family as there need be no relation between the growth of $r_n$ and the rate of convergence of the $\theta_n$. The $\omega_n$ converge slowly enough to account for both of these potentially distinct behaviors.

Next, we define two families of triangular regions in the plane by

\[
T_{1,n}=\{(x,y)\subset \R^2: r_n\geq x-\frac{r_n}{2}\geq \cot(\omega_n)|y|\},
\]
and
\[
T_{2,n}=\{(x,y):\frac{r_n(5+\sin(\omega_n))}{4}\geq x-\frac{r_n}{4}\geq \cot(\omega_n)|y|\}.
\]

\begin{figure}[h]
\centering
\begin{tikzpicture}[scale=2.5, >=stealth]

\draw[->] (-0.2,0) -- (3,0) node[right] {$x$};
\draw[->] (0,-1.2) -- (0,1.2) node[above] {$y$};

\def\rn{1.8}
\def\w{13} 
\pgfmathsetmacro{\lentwo}{1.25*\rn + \rn*sin(\w)/4}

\coordinate (A1) at ({\rn/2},0);
\coordinate (B1) at ({3*\rn/2}, {tan(\w)*\rn});
\coordinate (C1) at ({3*\rn/2},{-tan(\w)*\rn});
\coordinate (D1) at ({3*\rn/2},0);

\coordinate (A2) at ({\rn/4},0);
\coordinate (C2) at ({\rn/4+\lentwo}, {tan(\w)*\lentwo});
\coordinate (B2) at ({\rn/4+\lentwo}, -{tan(\w)*\lentwo});
\coordinate (D2) at ({\rn/4+\lentwo},0);

\fill[pattern=north east lines, pattern color=black!60] (A1) -- (B1) -- (C1) -- cycle;

\draw[thick, black] (A1) -- (B1) -- (C1) -- cycle node[midway, above, xshift=0.3cm, font=\LARGE] {$T_{1,n}$};
\draw[thick, black, dashed] (A2) -- (B2) -- (C2) -- cycle node[midway, above right, xshift=-0.3cm, font=\LARGE] {$T_{2,n}$};

\draw pic["$2\omega_n$", draw=black, angle radius=1cm, angle eccentricity=1.3, pic text options={xshift=0.2cm}, font=\LARGE] {angle = C1--A1--B1};

\node[below , yshift=-0.3cm, font=\large] at (A1) {$(\frac{r_n}{2},0)$};

\node[above left, yshift=0.1cm, xshift=0.1cm, font=\large] at (A2) {$(\frac{r_n}{4},0)$};

\node[xshift=-.6cm, yshift=0.3cm, font=\large] at (D1) {$(\frac{3r_n}{2},0)$};

\node[xshift=1.6cm, yshift=-0.3cm, font=\large] at (D2) {$(\frac{3r_n}{2}+\frac{r_n\sin(\omega_n)}{4},0)$};

\fill (A1) circle (0.5pt);
\fill (A2) circle (0.5pt);
\fill (B1) circle (0.5pt);
\fill (B2) circle (0.5pt);
\fill (C1) circle (0.5pt);
\fill (C2) circle (0.5pt);
\fill (D1) circle (0.5pt);
\fill (D2) circle (0.5pt);

\end{tikzpicture}
\caption{The two triangular regions $T_{1,n}$ and $T_{2,n}$}
\label{fig:placeholder}
\end{figure}
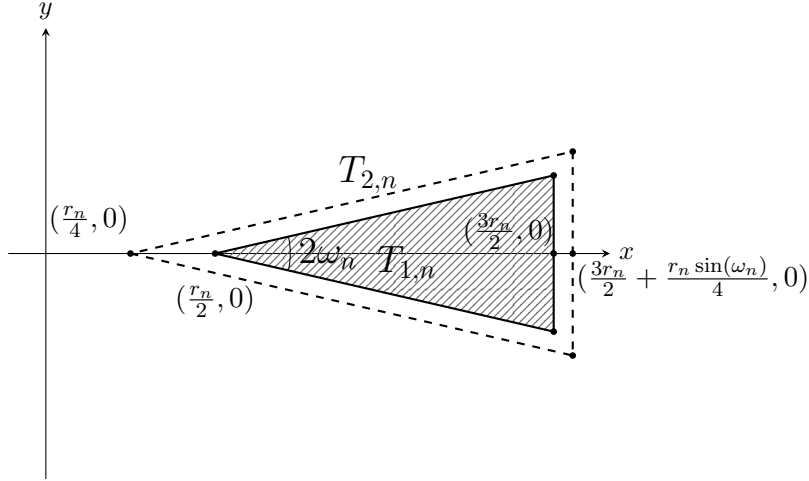
One can verify that by \eqref{E:SFAP2L_P2} that if $d=\text{dist}(T_{1,n},\R^2\setminus T_{2,n})$, then 
\[
d=\frac{r_n\sin(\omega_n)}{4}\geq \frac{r_n^\frac{p+1}{p+2}}{4}.
\]

We will then define the $\chi_n$ as smooth functions that satisfy
\[
\chi_n(x)=\begin{cases}
    1 & x+x_n\in T_{1,n}\\
    0 & x+x_n\not\in T_{2,n}
\end{cases}
\]
that additionally satisfy the bounds $|\partial^{\kappa}\chi_n|\lesssim r_n^{-|\kappa|\frac{p+1}{p+2}}$ for multiindicies $\alpha$, which is compatible with the value of $d$. We now verify \eqref{E:SFAP2L_C1}-\eqref{E:SFAP2L_C3}
\begin{proof}[Proof of \eqref{E:SFAP2L_C1}]
For $(x,y)\in \R^2$ we can guarantee that $\chi_n(x,y)=1$, provided that $(x+r_n\cos(\theta_n),y+r_n\sin(\theta_n))\in T_{1,n}$. Thus we must show that for $n$ sufficiently large the following two inequalities hold
\begin{equation}\label{E:TriangleEq1}
x+r_n\left(\cos(\theta_n)-\frac{1}{2}\right)\leq r_n
\end{equation}
and
\begin{equation}\label{E:TriangleEq2}
\cot(\omega_n)|y+r_n\sin(\theta_n)|\leq x+r_n\left(\cos(\theta_n)-\frac{1}{2}\right). 
\end{equation}
For \eqref{E:TriangleEq1} choosing $n$ large enough so that $r_n\geq2|x|$ one has
\[
x+r_n\left(\cos(\theta_n)-\frac{1}{2}\right)\leq|x|+\frac{r_n}{2}\leq r_n.
\]
To show \eqref{E:TriangleEq2} we see that choosing $n$ large enough so that $r_n\geq 8|x|$ and, recalling that $\theta_n\to 0$, additionally choosing $n$ large enough so that $\cos(\theta_n)>\frac{3}{4}$ one has
\[
x+r_n\left(\cos(\theta_n)-\frac{1}{2}\right)\geq -\frac{r_n}{8}+\frac{r_n}{4}=\frac{r_n}{8}.
\]
Similarly, choosing $n$ large enough so that $r_n\geq 16|y|^{\frac{p+2}{p+1}}$ one has by \eqref{E:SFAP2L_P2} and \eqref{E:SFAP2L_P3} that
\begin{align*}
\cot(\omega_n)|y+r_n\sin(\theta_n)|&\leq r_n^{\frac{1}{p+2}}|y|+r_n\left|\frac{\sin(\theta_n)}{\sin(\omega_n)}\right|\\
&\leq \frac{r_n}{16}+\frac{r_n}{16}=\frac{r_n}{8}.
\end{align*}
Together, the preceding two inequalities yield \eqref{E:TriangleEq2} establishing \eqref{E:SFAP2L_C1}.
\end{proof}
\begin{proof}[Proof of \eqref{E:SFAP2L_C2}]
For $\varepsilon>0$ we define $r_\varepsilon:\mathbb{S}^1\to [0,\infty)$ by 
\[
r_\varepsilon(\theta)=\inf\{r>0:|a(\rho\cos(\theta),\rho\sin(\theta))-\tilde{a}(\theta)|<\varepsilon\qtq{for all} \rho\geq r\}.
\]
Such a map exists by the repulsivity condition on $a$. Further, continuity of $\tilde{a}$ and uniform continuity of $a$ guarantee that $r_\varepsilon$ is continuous. Let $R=\max_{\theta\in\mathbb{S}^1}(r_\varepsilon(\theta))$. Now choose $n$ sufficiently large so that $r_n\geq R$ and, via another application of continuity of $\tilde{a}$ and \eqref{E:SFAP2L_P1}, so that 
\[
\sup_{\theta\in (-\omega_n,\omega_n)}|\tilde{a}(\theta)-\tilde{a}(0)|<\varepsilon.
\]
With these in place we see that for any point 
\[
z\in (r_n,\infty)\times(-\omega_n,\omega_n)\subset (0,\infty)\times \mathbb{S}^1
\]
an application of the triangle inequality shows that
\[
|a(z)-\tilde{a}(0)|<2\varepsilon.
\]
Finally, noting that
\[
\supp(\chi_n(\cdot-x_n))\subset T_{2,n}\subset (r_n,\infty)\times (-\omega_n,\omega_n),
\]
the claim follows.
\end{proof}
\begin{proof}[Proof of \eqref{E:SFAP2L_C3}]
Observe that
\[
\supp(\nabla \chi_n)\subset T_{2,n}\setminus T_{1,n},
\]
and likewise for $\supp(\Delta\chi_n)$. In particular, noting that 
\[
|T_{2,n}\setminus T_{1,n}|\lesssim r_n^2
\]
and recalling the symbol bounds $|\partial^\kappa\chi_n|\lesssim r_n^{-|\kappa|\frac{p+1}{p+2}}$,  H\"older's inequality yields the result.
\end{proof}
This completes the proof of the lemma.
\end{proof}

With this lemma in place we can establish the analogous result to Proposition~\ref{P:SFAP1} in the case that $p>2$. Because the details are quite similar to the previous case, we will only sketch the arguments.

\begin{proposition}[Scattering for far-away profiles: $p>2$]\label{P:SFAP2}
Let $p>2$ and let $a$ be admissible as in Definition~\ref{D:ADef}. Let $\varphi\in H^1$. Let $\{t_n\}\subset \R$ be such that $t_n\equiv0$ or $|t_n|\to\infty$ and let $\{(r_n\cos(\theta_n),r_n\sin(\theta_n)\}\subset \R^2$ satisfy $r_n\to \infty$ and $\theta_n\to \theta_\infty$. Then for $n$ sufficiently large there exists a global solution, $v_n$, to \eqref{nls} with
\[
v_n(0,x)=\varphi_n(x)=e^{it_n\Delta}\varphi(x-x_n)
\]
that obeys the bound
\[
\|v_n\|_{X(\R)}\lesssim_{\|\varphi\|_{H^1}}1.
\]
In particular, $v_n$ scatters both forward and backward in time. Furthermore, for any $\varepsilon>0$, there exists $N$ and $\psi\in C_c^\infty(\R\times\R^2)$ such that for $n\geq N$,
\[
\|v_n-\psi(\cdot+t_n,\cdot-x_n)\|_{X(\R)}<\varepsilon.
\]
\end{proposition}
\begin{proof}
Recall the definition of $\tilde{a}$ from Definition~\ref{D:ADef}. Because the $\theta_n\to \theta_\infty$ we expect that the appropriate limiting model in this scenario should be the defoucsing power-type NLS
\begin{equation}\label{nls2}
(i\partial_t+\Delta)u=\tilde{a}(\theta_\infty)|u|^pu.
\end{equation}
(In paticular if $\tilde{a}(\theta_\infty)=0$ it should be the free Schr\"odinger equation as in Proposition~\ref{P:SFAP1}).

If $t_n\equiv0$ we define $w(t,x)$ as the solution to \eqref{nls2} with initial data $w(0,x)=\varphi(x)$. If instead $t_n\to \infty$ (resp. $t_n\to-\infty)$ we define $w(t,x)$ as the solution to \eqref{nls2} that scatters to $\varphi$ as $t\to\infty$ (resp. $t\to -\infty$). From Nakanishi's work on the $1d$ and $2d$ defocusing power-type NLS in \cite{N}, we have that $w$ is global and obeys finite Strichartz norms. In particular by \eqref{E:SobforX} we have that
\[
\|w\|_{X(\R)}\lesssim_{\|\varphi\|_{H^1}}1,
\]
and $w$ scatters in both time directions.

We remark that if $\tilde{a}(\theta_\infty)=0$, then $w(t,x)$ is simply the solution to the free Schr\"odinger equation with initial data $w(0,x)=\varphi(x)$ regardless of behavior of $t_n$.

Let $\chi_n$ be the cutoff functions as in the preceding Lemma~\ref{L:Cutoffs}. Then for fixed $T>1$ we define our approximate solutions for $|t|<T$ by
\[
\tilde{v}_{n,T}(t,x)=\chi_n(x-x_n)w(t,x-x_n),
\]
and for $|t|>T$ we evolve freely so that
\[
\tilde{v}_{n,T}(t,x)=\begin{cases}
e^{i(t-T)\Delta}\tilde{v}_{n,T}(T) & t>T\\
e^{i(t+T)\Delta}\tilde{v}_{n,T}(-T) & t<T.
\end{cases}
\]
We are again are tasked with establishing \eqref{E:SFAP1_C1}-\eqref{E:SFAP1_C3} for the $\tilde{v}_{n,T}$. 

For \eqref{E:SFAP1_C1} the proof is identical in the case that the $t_n\equiv0$. In the case $t_n\to \infty$ then we estimate for $t_n>T$:
\begin{align*}
\|\tilde{v}_{n,T}(t_n)-\varphi_n(t_n)\|_{H^1}&\lesssim \|\chi_nw(T)-e^{iT\Delta}\varphi\|_{H^1}\\
&\lesssim\|(\chi_n-1)w(T)\|_{H^1}+\|w(T)-e^{iT\Delta}\varphi\|_{H^1},
\end{align*}
and the above tends to $0$ as $n,T\to\infty$ by \eqref{E:SFAP2L_C1}, $w(t)$ scattering to $\varphi$, and the Dominated Convergence Theorem.
The case of $t_n\to-\infty$ is similar.

The proof of \eqref{E:SFAP1_C2} is treated as in Proposition~\ref{P:SFAP1} using the $\|w\|_{X(\R)}$ bound previously discussed.

This leaves just \eqref{E:SFAP1_C3} which is more involved.

\begin{proof}[Proof of \eqref{E:SFAP1_C3}]
As in Proposition~\ref{P:SFAP1} we split into regions $|t|>T$ and $|t|\leq T$ with the region $|t|>T$ treated as in Proposition~\ref{P:SFAP1}. Hence, we focus on $|t|\leq T$ where the error
\[
e_{n,T}=(i\partial_t+\Delta)\tilde{v}_{n,T}-a(x)|\tilde{v}_{n,T}|^p\tilde{v}_{n,T}
\]
has the formula
\begin{align}
e_{n,T}&=\Delta\chi_n(x-x_n)w(t,x-x_n)\label{E:SFAP2_error1}\\
&+\nabla\chi_n(x-x_n)\cdot \nabla w(t,x-x_n)\label{E:SFAP2_error2}\\
&+\chi_n(x-x_n)\tilde{a}(\theta_\infty)|w(t,x-x_n)|^pw(t,x-x_n)\label{E:SFAP2_error3}\\
&-a(x)|\tilde{v}_{n,T}(t,x)|^p\tilde{v}_{n,T}(t,x)\label{E:SFAP2_error4}
\end{align}

The treatment of terms \eqref{E:SFAP2_error1} and \eqref{E:SFAP2_error2} follows as in Proposition~\ref{P:SFAP1} using \eqref{E:SFAP2L_C3}, so it remains to treat \eqref{E:SFAP2_error3} and \eqref{E:SFAP2_error4}. For these terms, we estimate via the triangle inequality and the nonlinear estimate \eqref{E:NonLinEst} on the time interval $[-T,T]$:
\begin{align*}
\|\eqref{E:SFAP2_error3}+\eqref{E:SFAP2_error4}\|_{Y}&\lesssim \|\chi_n^{p+1}(x-x_n)(\tilde{a}(\theta_\infty)-a(x))\|_{L^\infty}\|w\|_{X}^{p+1}\\
&\quad+\|(\chi_n-\chi_n^{p+1})\tilde{a}(\theta_\infty)w\|_{X}\|w\|_{X}^p\\
&\lesssim \|a(x)-\tilde{a}(\theta_\infty)\|_{L_x^\infty(\supp(\chi_n(x-x_n)))}\|w\|_{X}^{p+1}\\
&\quad+\|(\chi_n-\chi_n^{p+1})\tilde{a}(\theta_\infty)w\|_{X}\|w\|_{X}^p.
\end{align*}
The first summand vanishes as $n\to\infty$ by \eqref{E:SFAP2L_C2} with \eqref{E:SFAP1_C2} and the second summand vanishes as $n\to\infty$ by \eqref{E:SFAP2L_C1}, \eqref{E:SFAP1_C2}, and the Dominated Convergence Theorem.
\end{proof}

Having established \eqref{E:SFAP1_C1}-\eqref{E:SFAP1_C3} we apply stability to construct $v_n$ and prove approximation by $C_c^\infty(\R\times\R^2)$ functions as in Proposition~\ref{P:SFAP1}.
\end{proof}

\section{Reduction to Compact Solutions}\label{S:Reduction}

In this section we carry out the construction of a compact solution to \eqref{nls}, under the assumption that Theorem~\ref{T} were to fail. In this way the proof of Theorem~\ref{T} is reduced to precluding the existence of such solutions.

First, we define an increasing function
\[
L:[0,\infty)\to [0,\infty]\qtq{by} L(E)=\sup\|u\|_{X(\R)},
\]
where the supremum is taken over all solutions to \eqref{nls} such that $M(u)+E(u)< E$. By the small-data theory of Proposition~\ref{P:WellPosed}, we know that if $\eta_0$ is the small-data threshold then one has
\[
L(E)\lesssim_E1\qtq{for}E<\eta_0.
\]
Setting
\[
E_c=\sup\{E:L(E)<\infty\},
\]
to establish Theorem~\ref{T} it suffices to show that $E_c=\infty$. We will argue by contradiction, assuming that $E_c<\infty$. We first show the following:

\begin{proposition}[Palais-Smale Condition]\label{P:PSmale}
Suppose $E_c<\infty$. Let $u_n$ be a sequence of solutions to \eqref{nls} such that
\begin{equation}\label{E:PS_A1}
M(u_n)+E(u_n)\to E_c.
\end{equation}
Suppose that $\{t_n\}$ is a sequence of times such that
\begin{equation}\label{E:PS_A2}
\lim_{n\to\infty}\|u_n\|_{X(t_n,\infty)}=\lim_{n\to\infty}\|u_n\|_{X(-\infty,t_n)}=\infty.
\end{equation}
Then $\{u_n(t_n)\}$ converges along a subsequence in $H^1$.
\end{proposition}
\begin{proof}
By applying a time translation if necessary, we may assume that $t_n\equiv0$. Further, we may apply Proposition~\ref{P:LPD} to the sequence $\{u_n(0)\}$ so that after passing to a subsequence we have for any finite $J\leq J^*$
\[
u_n(0,x)=\sum_{j=1}^Je^{it_n^j\Delta}\varphi^j(x-x_n^j)+w_n^J,
\]
satisfying the properties of Proposition~\ref{P:LPD} with $e^{it\Delta}w_n^J$ vanishing in $X(\R)$. We will show that
\begin{equation}\label{E:PS_oneProfile}
\sup_{j}\limsup_{n\to\infty}\{M(\varphi^j)+E(\varphi^j)\}=E_c,
\end{equation}
which together with the mass/energy decoupling will guarantee that $J^*=1$ and $w_n^1\to 0$ in $H^1$ as $n\to\infty$.

To this end, we suppose that for some $\delta>0$ we have
\begin{equation}\label{E:PS_manyProfiles}
\sup_{j}\limsup_{n\to\infty}M(\varphi^j)+E(\varphi^j)\leq E_c-3\delta,
\end{equation}
in which case we will show that we can find a uniform bound on the $X$-norms for the $u_n$, contradicting \eqref{E:PS_A2}.

We now associate to each $\varphi^j$ a non-linear profile. If $(t_n^j,x_n^j)\equiv0$ then we define $v^j$ as the solution to \eqref{nls} with initial data $v^j(0,x)=\varphi^j$. In this case \eqref{E:PS_manyProfiles} guarantees that 
\[
M(v^j)+E(v^j)=M(\varphi^j)+E(\varphi^j)<E_c,
\]
so that $\|v^j\|_{X(\R)}<\infty$.

Similarly, if $t_n^j\to \infty$ (resp. $t_n^j\to -\infty$) and $x_n^j\equiv0$ we define $v^j$ as the solution to \eqref{nls} that scatters to $\varphi^j$ as $t\to\infty$ (resp. $t\to -\infty$). In this case, we can use the dispersive estimate and conservation of mass and energy to conclude that
\[
M(v^j)+E(v^j)\leq M(\varphi^j)+E(\varphi^j)+\delta<E_c,
\]
so that again $\|v^j\|_{X(\R)}<\infty$.
\\In either of these cases we then define $v_n^j(t,x)=v^j(t+t_n^j,x)$.

Finally, if $|x_n^j|\to \infty$ we construct $v_n^j$ by applying Proposition~\ref{P:SFAP1} or Proposition~\ref{P:SFAP2} depending on if $p\leq 2$ or $p>2$. We now define our approximate solutions to \eqref{nls} by setting
\[
u_n^J=\sum_{j=1}^Jv_n^j+e^{it\Delta}w_n^J. 
\]
By construction, we have that 
\[
\lim_{n\to\infty}\|u_n^J(0)-u_n(0)\|_{H^1}=0.
\]
We will additionally show that
\begin{equation}\label{E:PS_goodBound}
\limsup_{J\to J^*}\limsup_{n\to\infty} \|u_n^J\|_{X(\R)}\lesssim 1,
\end{equation}
and that
\begin{equation}\label{E:PS_errorBound}
\lim_{J\to J^*}\limsup_{n\to\infty}\|e_n^J\|_{Y(\R)}=0,
\end{equation}
where $e_n^J$ is the error
\[
e_J=(i\partial_t+\Delta)u_n^J-a(x)|u_n^J|^pu_n^J.
\]
Once we have established \eqref{E:PS_goodBound} and \eqref{E:PS_errorBound}, for $n,J$ sufficiently large we may apply stability, Proposition~\ref{P:Stability}, with the initial conditions $u_n(0)$, which will yield a uniform in $n$ bound on the $\|u_n\|_{X(\R)}$-norms, contradicting \eqref{E:PS_A2} and establishing \eqref{E:PS_oneProfile}. The key to proving \eqref{E:PS_goodBound} and \eqref{E:PS_errorBound} will be to exploit the asymptotic orthogonality of the parameters $(t_n^j,x_n^j)$ showing that in a suitable norm the product of any different $v_n^j,v_n^k$ will vanish.

Recalling the definition of $q,r$ from \ref{E:ExpDefsqr} (and in particular noting that $q,r>4$ for all choices of $p$) we define an auxillary function space
\[
X'(I)=L_t^\frac{q}{2}L_x^\frac{r}{2}(I\times\R^2).
\]
The importance of the $X'$-space is that for any $p>0$ an application of H\"older's inequality yields the following estimate
\begin{equation}\label{E:PS_nonlin1}
\|agh^p\|_{Y(\R)}\leq \begin{cases}
\|a\|_{L_x^\rho}\|gh\|_{X'(\R)}^p\|g\|_{X(\R)}^{1-p}& p\leq 1\\
\|a\|_{L_x^\rho}\|gh\|_{X'(\R)}\|h\|_{X(\R)}^{p-1} & p>1
\end{cases}
\end{equation}
for arbitrary functions $g,h:\R^2\to \C$. Then the precise form of orthogonality we will need is the following.

\begin{lemma}[Orthogonality]\label{L:Orthog}
For $j\not=k$ we have that
\[
\lim_{n\to\infty}\|v_n^jv_n^k\|_{X'(\R)}=0.
\]
\end{lemma}
The proof relies on approximation by $C_c^\infty$-functions (cf. \cite{KV, V}).

We are now in a position to prove \eqref{E:PS_goodBound} and \eqref{E:PS_errorBound}.
\begin{proof}[Proof of \eqref{E:PS_goodBound}]
By the vanishing property of the $e^{it\Delta}w_n^J$ it suffices to show that
\[
\limsup_{J\to J^*}\limsup_{n\to\infty}\left\|\sum_{j=1}^Jv_n^j\right\|_{X}\lesssim1.
\]
For any $\eta>0$ by $H^1$-decoupling we can find some $J_0=J_0(\eta)$ such that
\begin{equation}\label{E:PS_eta}
\sum_{j=J_0}^{J^*}\|\varphi^j\|_{H^1}^2<\eta.
\end{equation}
Choosing $\eta$ sufficiently small we may then appeal to the small-data theory of Proposition~\ref{P:WellPosed} to obtain
\begin{equation}\label{E:PS_smallDataBound}
\sum_{j=J_0}^{J^*}\|v_n^j\|_{X(\R)}^2\lesssim \eta.
\end{equation}
We now estimate for any $J>J_0$ via the triangle inequality, \eqref{E:PS_smallDataBound}, and Lemma~\ref{L:Orthog}
\begin{align*}
\left\|\sum_{j=1}^Jv_n^j\right\|_{X}^2&\lesssim\sum_{j=1}^{J_0}\|v_n^J\|_{X}^2+\sum_{j={J_0}}^J\|v_n^j\|_X^2+C(J)\sum_{j\not=k}\|v_n^jv_n^k\|_{X'}\\
&\lesssim\sum_{j=1}^{J_0}\|v_n^j\|_{X}^2+\eta+C(J)o_n(1),
\end{align*}
which is acceptable.
\end{proof}
\begin{proof}[Proof of \eqref{E:PS_errorBound}]
Let us once again use the notation $f(z)=|z|^pz$ and $F(x,z)=a(x)f(z)$, where again the $x$-dependence will be suppressed. We first rewrite the error as 
\begin{align}
e_n^J&=\sum_{j=1}^JF(v_n^j)-F\left(\sum_{j=1}^Jv_n^j\right)\label{E:PS_error1}\\
&+F(u_n^J-e^{it\Delta}w_n^J)-F(u_n^J)\label{E:PS_error2}.
\end{align}
Now we estimate each of these terms separately relying on pointwise estimates. For \eqref{E:PS_error1}, we estimate using the pointwise estimate \eqref{E:PW_2} and \eqref{E:PS_nonlin1}    
\begin{align*}
\|\eqref{E:PS_error1}\|_{Y}&\lesssim_J\sum_{j\not=k}\|a\,|v_n^j|\,|v_n^k|^p\|_{Y}\\
&\lesssim_J\begin{cases}
    \displaystyle\sum_{j\not=k}\|a\|_{L_x^\rho}\|v_n^jv_n^k\|_{X'}^p\|v_n^j\|_X^{1-p}&p\leq 1\\
    \displaystyle\sum_{j\not=k}\|a\|_{L_x^\rho}\|v_n^jv_n^k\|_{X'}\|v_n^k\|_{X}^{p-1}&p>1
\end{cases}
\end{align*}
which for fixed $J$ tends to $0$ as $n\to \infty$ by \eqref{E:PS_goodBound} and orthogonality, Lemma~\ref{L:Orthog}. For \eqref{E:PS_error2}, we estimate using the pointwise estimate \eqref{E:PW_1} and the nonlinear estimate \eqref{E:NonLinEst}
\begin{align*}
\|\eqref{E:PS_error2}\|_{Y}&\lesssim \|a\|_{L_x^\rho}\|e^{it\Delta}w_n^J\|_{X}\left(\|e^{it\Delta}w_n^J\|_{X}^p+\|u_n^J\|_{X}^p\right),
\end{align*}
which again tends to $0$ as $n\to\infty$ and $J\to J^*$ by \eqref{E:PS_goodBound} and the vanishing property of the $e^{it\Delta}w_n^J$.
\end{proof}
Having established \eqref{E:PS_goodBound} and \eqref{E:PS_errorBound}, we may conclude that $J^*=1$ and that $w_n^1\to 0$ in $H^1$ as $n\to\infty$. To conclude the proof we need to show that the parameters $(t_n^1,x_n^1)\equiv0$. 

If the $|x_n^1|\to \infty$, then we may quickly reach a contradiction by applying one of Proposition~\ref{P:SFAP1} or Proposition~\ref{P:SFAP2} depending on the value of $p$. If the $|t_n^1|\to\infty$ then setting
\[
\tilde{u}_n(t)=e^{it\Delta}u_n(0),
\]
one can show that the $\tilde{u}_n$ form good approximate solutions to \eqref{nls} with uniformly bounded $X$-norms for sufficiently large $n$, and apply stability to reach a contradiction. We omit the standard details.
\end{proof}

Having established the preceding result, we can quickly establish the construction of a compact solution.

\begin{proposition}[Reduction to Compact Solutions]\label{P:Reduct}
If Theorem~\ref{T} fails, then there exists a solution $v:\R\times\R^2\to \C$ to \eqref{nls} with $v(0,x)\in H^1$ such that $\{v(t):t\in \R\}$ is pre-compact in $H^1$.
\end{proposition}
\begin{proof}
Suppose that Theorem~\ref{T} fails, so that $E_c<\infty$. Then by definition, we may find a sequence of solutions $u_n$ to \eqref{nls} such that
\[
M(u_n)+E(u_n)\nearrow E_c\qtq{and} \lim_{n\to\infty}\|u_n\|_{X(\R)}=\infty.
\]
In particular, we may find a sequence of times $t_n$ such that
\[
\lim_{n\to\infty}\|u_n\|_{X(-\infty,t_n)}=\lim_{n\to\infty}\|u_n\|_{X(t_n,\infty)}=\infty.
\]
After applying a time-translation we may further assume that $t_n\equiv0$.
Applying the Palais-Smale condition Proposition~\ref{P:PSmale} we see that there exists some $\varphi\in H^1$ such that $u_n(t_n)\to \varphi$ in $H^1$ as $n\to\infty$. Define $v$ as the solution to \eqref{nls} with initial data $v(0,x)=\varphi$. By stability, we have that
\[
M(v)+E(v)=E_c\qtq{and}\|v\|_{X(-\infty,0)}=\|v\|_{X(0,\infty)}=\infty.
\]
We now seek to show that $\{v(t):t\in \R\}$ is pre-compact in $H^1$. Letting $\{s_n\}\subset\R$ be an arbitrary sequence of real numbers we see that
\[
\lim_{n\to\infty}\|v\|_{X(-\infty,s_n)}=\lim_{n\to\infty}\|v\|_{X(s_n,\infty)}=\infty.
\]
Thus another application of Proposition~\ref{P:PSmale} implies that $v(s_n)$ converges along a subsequence in $H^1$, yielding the desired pre-compactness.
\end{proof}

\section{Preclusion of Compact Solutions}\label{S:Preclusion}

In this section, we quickly preclude the existence of compact solutions as in Proposition~\ref{P:Reduct} concluding the proof of Theorem~\ref{T}.

\begin{definition}
For $g\in H^1$ and $t\geq 0$ we define a norm
\[
\|g\|_{Z(t)}=\left\|\frac{t^\frac{1}{2}(x+2it\nabla)f}{(\langle t\rangle^3+|x|^3)^\frac{1}{2}}\right\|_{L_x^2(\R^2)}.
\]
\end{definition}
An application of H\"older's inequality shows that for all $t\geq 1$ we have $\|g\|_{Z(t)}\lesssim \|g\|_{H^1}$ uniformly in $t$.
The identity
\[
(x+2it\nabla)g=\exp\left[\frac{i|x|^2}{4t}\right]2it\nabla\exp\left[\frac{-i|x|^2}{4t}\right]g,
\]
shows that for all $g\in H^1(\R^2)$, if $\|g\|_{Z(t)}=0$ then $g\equiv0$. In \cite{N} Nakanishi established a Morawetz estimate applicable to the $1d$ and $2d$ defocusing power-type NLS involving this $Z(t)$-norm. In \cite{BM} the authors adapted this estimate to the case of \eqref{nls} relying cruically on the repulsivity of $a$. The exact estimate is as follows:
\begin{proposition}[Morawetz Estimate \cite{BM,N}]\label{P:Mora}
For $u$ a solution to \eqref{nls} one has
\[
\int_1^\infty\|u(t)\|_{Z(t)}^2\,\frac{dt}{t}\lesssim_{M(u)+E(u)}1.
\]
\end{proposition}

To reach our desired contradiction, we additionally prove the following.

\begin{proposition}\label{P:Preclusion}
Let $u$ be a nonzero compact solution to \eqref{nls} as in Proposition~\ref{P:Reduct}. Then
\[
\inf_{t\geq 1}\|u(t)\|_{Z(t)}\gtrsim1.
\]
\end{proposition}
\begin{proof}
Suppose that $\{t_n\}\subset\R$ is a sequence of times such that
\[
\|u(t_n)\|_{Z(t_n)}\to 0.
\]
By passing to a subsequence we may guarantee that either $t_n\to t_\infty\in [1,\infty)$ or $t_n\to\infty$. Further, because $\{u(t):t\in\R\}$ is a pre-compact set, after passing to a subsequence we have that $u(t_n)\to \varphi$ in $H^1$ for some $\varphi\in H^1$. Now if $t_n\to t_0$ because $u\in C_t([0,\infty),H_x^1(\R^2))$, we must have that $\varphi=u(t_0)$. But now by continuity of the norm,
\[
\|u(t_0)\|_{Z(t_0)}=0\qtq{implies}u(t_0)\equiv0,
\]
and by uniqueness we have that $u\equiv0$, a contradiction. 

Similarly, if $t_n\to\infty$ then we estimate 
\begin{align*}
\|\varphi\|_{Z(t_n)}&\lesssim\|u(t_n)\|_{Z(t_n)}+\|\varphi-u(t_n)\|_{Z(t_n)}\\
&\lesssim \|u(t_n)\|_{Z(t_n)}+\|\varphi-u(t_n)\|_{H^1},
\end{align*}
which tends to $0$ as $n\to\infty$. However, because the $t_n\to\infty$ we also have by the dominated convergence theorem that
\[
\|\varphi\|_{Z(t_n)}=\|2\nabla \varphi\|_{L^2}+o_n(1),
\]
and so we may conclude that $\varphi\equiv0$, which again by uniqueness implies that $u\equiv0$, again a contradiction.
\end{proof}
\begin{corollary}
Combining the preceding Proposition with Proposition~\ref{P:Mora}, we see that there are no nonzero compact solutions to \eqref{nls}, proving Theorem~\ref{T}.
\end{corollary}


\begin{thebibliography}{100}


\bibitem{AGT} L. Aloui, M. Grira, and S. Tayachi, \emph{Scattering results for the inhomogeneous nonlinear Schrödinger equation.}  J. Math. Anal. Appl. \textbf{548} (2025), no. 1, Paper No. 129368, 20 pp.

\bibitem{BM} L. Baker and J. Murphy, \emph{Scattering for the $1d$ NLS with inhomogeneous nonlinearities.} Preprint {\tt arXiv:2509.13438}.

\bibitem{Bourgain} J. Bourgain, \emph{Global wellposedness of defocusing critical nonlinear Schr\"odinger equation in the radial case.} J. Amer. Math. Soc. \textbf{12} (1999), no. 1, 145--171.

\bibitem{BGTV} N. Burq, V. Georgiev, N. Tzvetkov, and N. Visciglia, \emph{Scattering for mass-subcritical NLS with short-range nonlinearity and initial data in $\Sigma$}. Ann. Henri Poincaré \textbf{24} (2023), no. 4, 1355--1376.

\bibitem{CC} L. Campos and M. Cardoso, \emph{A virial-Morawetz approach to scattering for the non-radial inhomogeneous NLS.}  Proc. Amer. Math. Soc. \textbf{150} (2022), no. 5, 2007--2021.

\bibitem{CFGM} M. Cardoso, L.~G. Farah, C.~M. Guzm\'an, and J. Murphy, \emph{Scattering below the ground state for the intercritical non-radial inhomogeneous NLS.} Nonlinear Anal. Real World Appl. \textbf{68} (2022), Paper No. 103687, 19 pp.

\bibitem{C} T. Cazenave, \emph{Semilinear Schrödinger equations}. Courant Lect. Notes Math., \textbf{10}. New York University, Courant Institute of Mathematical Sciences, New York; American Mathematical Society, Providence, RI, 2003.

\bibitem{CW} T. Cazenave and F. Weissler, \emph{Rapidly decaying solutions of the nonlinear Schr\"odinger equation.} Comm. Math. Phys. \textbf{147} (1992), no. 1, 75--100.

\bibitem{CGT} J. Colliander, M. Grillakis, and N. Tzirakis, \emph{Tensor products and correlation estimates with applications to nonlinear Schr\"odinger equations.} Comm. Pure Appl. Math. {\bf 62} (2009), no.~7, 920--968.

\bibitem{CHVZ}J. Colliander, J. Holmer, M. Visan, and X. Zhang, \emph{Global existence and scattering for rough solutions to generalized nonlinear Schrodinger equations on $\mathbb{R}$}. Commun. Pure Appl. Anal. \textbf{7} (2008), no. 3, 467--489.

\bibitem{CKSTT}J. Colliander, M. Keel, G. Staffilani, H. Takaoka, T. Tao, \emph{Global well-posedness and scattering for the energy-critical nonlinear Schr\"odinger equation in $\R^3$.} Ann. of Math. (2) {\bf 167} (2008), no.~3, 767--865.

\bibitem{CS} P. Constantin and J.-C. Saut, \emph{Local smoothing properties of dispersive equations.} J. Amer. Math. Soc. 1 (1988), 413--439.

\bibitem{CLZ} Z. Cui, Y. Li, and D. Zhao, \emph{Decay of solutions to one-dimensional inhomogeneous nonlinear Schrödinger equations}. Preprint {\tt arXiv:2412.08272}.

\bibitem{CLZ2} Z. Cui, Y. Li, and D. Zhao, \emph{Well-posedness and scattering of odd solutions for the defocusing INLS in one dimension}. Preprint {\tt arXiv:2509.02158}.

\bibitem{D} B. Dodson, \emph{Global well-posedness and scattering for the defocusing, $L^{2}$-critical nonlinear Schr\"odinger equation when $d\geq3$.} J. Amer. Math. Soc. {\bf 25} (2012), no.~2, 429--463.

\bibitem{D2} B. Dodson, \emph{Global well-posedness and scattering for the defocusing, $L^2$-critical, nonlinear Schr\"odinger equation when $d=2$.} Duke Math. J. {\bf 165} (2016), no.~18, 3435--3516.

\bibitem{D3} B. Dodson, \emph{Global well-posedness and scattering for the defocusing, $L^2$ critical, nonlinear Schr\"odinger equation when $d=1$.} Amer. J. Math. {\bf 138} (2016), no.~2, 531--569.

\bibitem{FXC} D. Fang, J. Xie, and T. Cazenave, \emph{Scattering for the focusing energy-subcritical nonlinear Schrödinger equation.} Sci. China Math. \textbf{54} (2011), no. 10, 2037--2062.

\bibitem{FarahG} L.~G. Farah and C.~M. Guzm\'an, \emph{Scattering for the radial 3D cubic focusing inhomogeneous nonlinear Schr\"odinger equation.} J. Differential Equations {\bf 262} (2017), no. 8, 4175--4231.

\bibitem{GV} J. Ginibre and G. Velo, \emph{Smoothing properties and retarded estimates for some dispersive evolution equations}. Comm. Math. Phys. \textbf{144} (1992), 163–-188.

\bibitem{GM} C.~M. Guzm\'an and J. Murphy, \emph{Scattering for the non-radial energy-critical inhomogeneous NLS.} J. Differential Equations \textbf{295} (2021), 187--210.

\bibitem{HKV} B. Harrop-Griffiths, R. Killip, and M. Vi\c san, \emph{Scattering for the nonlinear Schr\" odinger equation with concentrated nonlinearity.} Preprint {\tt arXiv:2507.14571v1}.

\bibitem{HR} J. Holmer and S. Roudenko, \emph{A sharp condition for scattering of the radial 3D cubic nonlinear Schr\"odinger equation}. Comm. Math. Phys. {\bf 282} (2008), no. 2, 435--467.

\bibitem{KT} M. Keel and T. Tao, \emph{Endpoint Strichartz Estimates}. Amer. J. Math. \textbf{120} (1998), no. 5, 955–980.

\bibitem{KM} C.~E. Kenig and F. Merle, \emph{Global well-posedness, scattering and blow-up for the energy-critical, focusing, non-linear Schr\"odinger equation in the radial case.} Invent. Math. {\bf 166} (2006), no.~3, 645--675.

\bibitem{Keraani} S. Keraani, \emph{On the defect of compactness for the Strichartz estimates for the Schr\"odinger equations.} J. Diff. Eq. 175 (2001), 353--392.

\bibitem{KMVZZ} R. Killip, C. Miao, M. Vi\c san, J. Zhang, and J. Zheng, \emph{The energy-critical NLS with inverse-square potential.} Discrete Contin. Dyn. Syst. \textbf{37} (2017), no. 7, 3831--3866.

\bibitem{KMVZ} R. Killip, J. Murphy, M. Vi\c san, and J. Zheng, \emph{The focusing cubic NLS with inverse-square potential in three space dimensions.} Differential Integral Equations \textbf{30} (2017), no. 3-4, 161–206.

\bibitem{KV} R. Killip and M. Vi\c san, \emph{Nonlinear Schrödinger equations at critical regularity.} Lecture notes prepared for Clay Mathematics Institute Summer School, Z\"urich, Switzerland, 2008.

\bibitem{LS} H. Linblad and A. Soffer, \emph{Scattering and small data completeness for the critical nonlinear Schrödinger equation}. Nonlinearity \textbf{19} (2006), no. 2, 345–353.

\bibitem{LMZ} X. Liu, C. Miao, and J. Zheng, \emph{Global well-posedness and scattering for mass-critical inhomogeneous NLS when $d\geq 3$}. Math. Z. \textbf{311} (2025), no. 3, Paper No. 54.

\bibitem{MMZ} C. Miao, J. Murphy, and J. Zheng, \emph{Scattering for the non-radial inhomogeneous NLS.} Math. Res. Lett. \textbf{28} (2021), no. 5, 1481–1504.

\bibitem{M-expo} J. Murphy, \emph{Subcritical scattering for defocusing nonlinear Schr\"odinger equations}. Available online at https://pages.uoregon.edu/jamu/expository.pdf.

\bibitem{M-INLS} J. Murphy, \emph{A simple proof of scattering for the intercritical inhomogeneous NLS.} Proc. Amer. Math. Soc. \textbf{150} (2022), no. 3, 1177--1186.  

\bibitem{M-half} J. Murphy, \emph{The defocusing $\dot H^{1/2}$-critical NLS in high dimensions}. Discrete Contin. Dyn. Syst. {\bf 34} (2014), no. 2, 733--748.

\bibitem{N} K. Nakanishi, \emph{Energy scattering for nonlinear Klein-Gordon and Schrödinger equations in spatial dimensions 1 and 2.} J. Funct. Anal. \textbf{69} (1999), no. 1, 201–225.

\bibitem{PV} F. Planchon and L. Vega, \emph{Bilinear virial identities and applications.} Ann. Sci. \'Ec. Norm. Sup\'er. (4) \textbf{42} (2009), no. 2, 261--290.

\bibitem{RV} E. Ryckman, M. Visan, \emph{Global well-posedness and scattering for the defocusing energy-critical nonlinear Schr\"odinger equation in $\R^{1+4}$.} Amer. J. Math. {\bf 129} (2007), no.~1, 1--60.

\bibitem{Sjo} P. Sj\"olin, \emph{Regularity of solutions to the Schr\"odinger equation.} Duke Math. J. \textbf{55} (1987), 699--715

\bibitem{S} W. A. Strauss, \emph{Nonlinear scattering theory.} In Scattering Theory in Mathematical Physics, edited by J. A. Lavita and J. P. Marchand. D. Reidel, Dordrecht, Holland/Boston, 1974, pp. 53–178.

\bibitem{S2} R. S. Strichartz, \emph{Restrictions of Fourier transforms to quadratic surfaces and decay of solutions of wave equations}. Duke Math. J. \textbf{44} (1977), no. 3, 705–714.

\bibitem{TV}T. Tao and M. Visan, \emph{Stability of energy-critical nonlinear Schr\"odinger equations in high dimensions}. Electron. J. Differential Equations {\bf 118} (2005), 1-28.

\bibitem{TY}Y. Tsutsumi and K. Yajima, \emph{The Asymptotic Behavior of Nonlinear Schrödinger Equations}. Bull. Amer. Math. Soc. (N.S.) \textbf{11} (1984), no. 1, 186–188.

\bibitem{Vega} L. Vega, \emph{Schr\"odinger equations: pointwise convergence to the initial data.} Proc. Amer. Math. Soc. 102 (1988), 874--878.

\bibitem{V} M. Vi\c san, \emph{Dispersive Equations}. In Dispersive equations and nonlinear waves. Oberwolfach Seminars, 45. Birkh\"auser/Springer, Basel, 2014. xii+312 pp.

\bibitem{V2} M. Visan, \emph{The defocusing energy-critical nonlinear Schr\"odinger equation in higher dimensions.} Duke Math. J. {\bf 138} (2007), no.~2, 281--374.

\bibitem{W} M. Watanabe, \emph{Time-dependent method for non-linear Schrödinger equations in inverse scattering problems.} J. Math. Anal. Appl. \textbf{459} (2018), no. 2, 932–944.

\end{thebibliography}
\end{document}